\documentclass[11pt]{amsart} %AA
\usepackage{amsmath,amssymb,latexsym, amscd}%amsthm,
\usepackage[all]{xy}
\usepackage{todonotes}

\usepackage{hyperref}

\definecolor{dark_purple}{rgb}{0.4, 0.0, 0.4}
\definecolor{dark_green}{rgb}{0.0, 0.7, 0.0}

\def\H{{\mathcal H}}
\def\HH{{\underline{\mathcal H}}} %JCW
\def\N{{\mathbb N}}

\def\C{{\mathbb C}}
\def\R{{\mathbb R}}
\def\HP{{\mathbb H}P}
\def\Z{{\mathbb Z}}

\def\Cc{{\mathbf C}} %JCW
\def\CC{\underline{\mathbb C}} %Redefined JCW
\def\cc{{\mathbf c}} %JCW
\def\DD{{\mathbf D}} %JCW
\def\EE{{\mathbf E}} %JCW
\def\FF{{\mathbf F}} %JCW
\def\GG{{\mathbf G}} %JCW
\def\BB{{\mathbf B}} %JCW
\def\ee{{\mathbf e}} %JCW
\def\ep{w} %JCW
\def\epp{\mathbf w} %JCW

\def\v{\varphi}
\newcommand{\Ker}[1]{\mathsf{Ker}~ }

\newcommand{\codim}[1]{\mathsf{codim}~ }

\newtheorem{theorem}{Theorem}[section]
\newtheorem{proposition}  [theorem]  {Proposition} %JCW
\newtheorem{lemma}        [theorem]  {Lemma} %JCW
\newtheorem{example}      [theorem]  {Example} %JCW
\newtheorem{remark}		  [theorem]  {Remark} %JCW To make numbering sequential
\newcommand{\wt}{\widetilde}

\DeclareMathOperator{\image}{Im} %Image
\DeclareMathOperator{\U}{U} %unitary group
\DeclareMathOperator{\spa}{span} %span

\DeclareMathOperator{\rk}{rank} %JCW
\DeclareMathOperator{\antidiag}{antidiag} %JCW

\newcommand{\ii}{\mathrm{i}} %sqrt{-1}
\newcommand{\gl}{\mathfrak{gl}}

\newcommand{\CP}{{\mathbb C}P}
\newcommand{\RP}{{\mathbb R}P}

\newcommand{\pa}{\partial}

\newcommand{\zbar}{\bar{z}}

\newcommand{\ov}{\overline}

\newcommand{\al}{\alpha}
\newcommand{\be}{\beta}
\newcommand{\ga}{\gamma}
\newcommand{\la}{\lambda}

\allowdisplaybreaks

\begin{document}

\title[Diagrams and harmonic maps, revisited]{Diagrams and harmonic maps, revisited}
\author[R. Pacheco]{Rui Pacheco}
\address{Centro de Matem\'{a}tica e Aplica\c{c}{\~{o}}es (CMA-UBI), Universidade da Beira Interior, 6201 -- 001
Covilh{\~{a}}, Portugal.}
\email{rpacheco@ubi.pt}
\thanks{The first author was partially supported by Funda\c{c}\~{a}o para a Ci\^{e}ncia e Tecnologia through the project UID/MAT/00212/2019.}
\author[J.C. Wood]{John C. Wood}
\address{School of Mathematics, University of Leeds, LS2 9JT, G.B.}
\email{j.c.wood@leeds.ac.uk}
\thanks{}


\begin{abstract} We extend many known results for harmonic maps from  the 2-sphere into a Grassmannian  to harmonic maps of finite uniton number from an arbitrary Riemann surface. Our method relies on  a new theory of nilpotent cycles arising from the diagrams  of F.E.~Burstall and the second author associated to such harmonic maps; these properties arise from a criterion for finiteness of the uniton number  found recently by the authors with A.~Aleman.  Applications include a new classification result on minimal surfaces of constant curvature and a constancy result for finite type harmonic maps.
\end{abstract}

\maketitle

\section{Introduction and Preliminaries} \label{sec:intro}
\subsection{Results}
Recall that a smooth map $\v$ between two Riemannian manifolds $(M,g)$ and $(N,h)$ is said to be \emph{harmonic} if it is a critical point of the energy functional
$$ E(\v, D)=\frac{1}{2}\int_D|d\v|^2\omega_g$$
for any relatively compact $D$ in $M$, where $\omega_g$ is the volume measure, and $|d\v|^2 $ is the Hilbert--Schmidt  norm of the differential  of $\v$.

An important class of harmonic maps is those from a Riemann surface to a Lie group or symmetric space.   In particular K.\ Uhlenbeck \cite{uhlenbeck} defined the notion of \emph{uniton number} for harmonic maps to the unitary group and showed it to be finite for maps from the Riemann sphere $S^2$ --- see G.~Valli \cite{valli} for a simpler proof of this.  The finiteness has huge consequences for the structure of those harmonic maps, developed in many papers since then.  In particular, all such maps are (weakly) conformal which means that they give \emph{minimal branched immersions} in the sense of \cite{G-O-R}.  It is thus important to know whether a harmonic map from an arbitrary Riemann surface is of finite uniton number.
In \cite{APW1}, the present authors and A.\ Aleman developed a useful criterion for this  finiteness; the criterion uses the extended solutions of Uhlenbeck and their interpretation by G.\ Segal \cite{segal} as varying shift-invariant subspaces of a Hilbert space,  but gives a test which depends only on the harmonic map, and is local, see Proposition \ref{S-stabilizes} below.

For maps to a (complex) Grassmannian $G_k(\C^n)$, this \emph{finiteness criterion} can be interpreted in terms of \emph{diagrams  for harmonic maps} developed by F.~E.\ Burstall and the second author \cite{burstall-wood}.
In this paper we study and apply the criterion to extend many known results for harmonic maps from the Riemann sphere to harmonic maps of finite uniton number from an arbitrary Riemann surface. 
 Note that our methods do not apply, in general, to harmonic maps which are \emph{strongly isotropic} i.e., \emph{of infinite isotropy order}; however these are automatically of finite uniton number and are well understood, see \cite{erdem-wood, burstall-wood}.

(i) We extend descriptions of harmonic maps in \cite{burstall-wood} and \cite{bahy-wood-G2} to show how harmonic maps of finite uniton number from any Riemann surface into $G_2(\C^n)$ can be obtained from three basic types (Theorem \ref{BuWo 3.3}).  This restricts to descriptions of maps into
 a real Grassmannian $G_2(\R^n)$ and  quaternionic projective space  $\HP^n$ extending \cite{bahy-wood-G2} and \cite{bahy-wood-HPn}. Our work depends on showing that  certain \emph{cycles} (or \emph{circuits}) of second fundamental forms in a diagram associated to the harmonic map, whose composition is nilpotent on $S^2$, because of  the vanishing of holomorphic differentials on $S^2$, \emph{remain nilpotent for harmonic maps of finite uniton number from any Riemann surface, as a result of the finiteness criterion}.
Our results give extensions to such harmonic maps of theorems of J.~Ramanathan \cite{ramanathan} and A.~R.\ Aithal \cite{aithal-G25} for harmonic maps into $G_2(\C^4)$ and $G_2(\C^5)$, Aithal for $\HP^2$ \cite{aithal-HP2}, and A.~Bahy-El-Dien and the second author for $G_2(\R^n)$ (or the complex quadric $Q_{n-2}$) \cite{bahy-wood-G2} and for $\HP^n$ \cite{bahy-wood-HPn}.

(ii) We develop the construction of such \emph{nilpotent cycles} to other Grassmannians.  This leads to two families of such cycles (Propositions \ref{ce^s} and \ref{nilorder p}), plus two `hybrid' cycles (Propositions \ref{cece^2} and \ref{c^2ece}), which we hope will be the start of a larger theory of  nilpotent cycles.   We use our nilpotent cycles to
extend the description \cite{burstall-wood} of harmonic maps from $S^2$ to $G_k(\C^n)$ ($k=3,4,5$) to finite uniton number harmonic maps from any Riemann surface (Theorem \ref{BuWo 4.2}).

(iii) A second important type of harmonic map is those of \emph{finite type}. We show that \emph{a harmonic map from a torus to a complex Grassmannian which is simultaneously of finite uniton number and finite type is constant} (Theorem \ref{finite type constant}).  This was previously known only for maps into a sphere or complex projective space \cite{pacheco-tori}.  We prove it by extending a theorem of J.~Wolfson \cite{wolfson} to show that \emph{the harmonic sequence of a harmonic map into a Grassmannian terminates if and only if it is of finite uniton number} (Theorem \ref{th:wolfson}).

 (iv) Although more recent work has shown how harmonic maps of finite uniton number can be constructed explicitly from holomorphic data, the above methods are more suited to answering geometrical questions such as finding \emph{constant curvature or homogeneous minimal surfaces}, see for example \cite{FeiHe,HeJiao2014, HeJiao,JiaoLi Qn, JiaoLi Q5,JiaoXu, LiHeJiaoQ3, PengWangXu}. These papers took the domain to be $S^2$, but the methods in this paper allow some generalization of these results to harmonic maps of finite uniton number from any Riemann surface, see \S \ref{subsec:const-curv}.

\subsection{Background}\label{subsec:background}
We give only the background needed for this paper.  For  more details on  the finiteness criterion plus general background from a functional analytic point of view see \cite{APW1}; see also
\cite{eells-lemaire, urakawa} for the general theory  of harmonic maps and \cite{svensson-wood-unitons, wood-60} for some theory relevant to this paper.

 \emph{Throughout the rest of this paper, $M$ will denote a (connected) Riemann surface, not necessarily compact.}
For a smooth map $\v:M \to \U(n)$,
we consider the matrix-valued 1-form
\begin{equation*}
\tfrac{1}{2}\v^{-1}d\v := A^\v_z d z+A^\v_{\bar{z}}d\bar{z},
\end{equation*}
where $z$ is a local (complex) coordinate on $M$. Then (see \cite{uhlenbeck}) \emph{$\v$  is harmonic if and only if}
\begin{equation*}
(A^\v_z)_{\bar{z}}+(A^\v_{\bar{z}})_z=0.
\end{equation*}

Recall \cite{uhlenbeck} that an \emph{extended solution} is a smooth map $\Phi:S^1\times M\to \U(n)$ satisfying $\Phi(1,\cdot)=I$ and such that, for every local (complex) coordinate $z$ on $M$, there are $\gl(n,\C)$-valued maps $A_z$ and $A_{\bar{z}}$ for which
\begin{equation}\label{extsol}
\Phi(\la,\cdot)^{-1}d\Phi(\la,\cdot)=(1-\la^{-1})A_z d z+(1-\la)A_{\bar{z}}d\bar{z}.
\end{equation}
We can consider $\Phi$ as a map from $M$ into the \emph{loop group} of $\U(n)$ defined by
$\Omega\U(n) = \{\gamma:S^1 \to \U(n) \text{ smooth}: \ga(1) = I\}$. If $\Phi$ is an extended solution, then
 $\v=\Phi(-1,\cdot)$ is a harmonic map with (*) $A^\v_z=A_z$ and $A^\v_{\bar{z}}=A_{\bar{z}}$.
 Conversely, given a harmonic map $\v:M\to \U(n)$, an extended solution  $\Phi$
is said to be \emph{associated} to $\v$ if
$\Phi(-1,\cdot)=u\v$ for some constant $u\in \U(n)$, equivalently, $\Phi$ satisfies (*).
If $M$ is simply connected, the existence of an extended solution $\Phi$ associated to a smooth map is equivalent to harmonicity of the map, see \cite{uhlenbeck}; $\Phi$ is then unique up to multiplication from the left by a \emph{constant loop}, i.e., a $\U(n)$-valued function on $S^1$, independent of $z\in M$.

Let $\v:M \to \U(n)$ be a  smooth map.
Recall that the \emph{(Koszul--Malgrange) holomorphic structure induced by $\v$} is the unique holomorphic structure on the trivial bundle
$\CC^n := M \times \C^n$ with $\bar{\pa}$-operator given locally by the derivation
$D^{\v}_{\bar z} = \pa_{\zbar} + A^{\v}_{\zbar}$;  we denote the resulting holomorphic vector bundle by $(\CC^n, D^{\v}_{\bar z})$.
Here $\pa_{\zbar}$ denotes ordinary (`flat') differentiation with respect to $\zbar$ in $\C^n$.
Note that \cite{uhlenbeck} \emph{a map $\v:M\to \U(n)$ is harmonic if and only if
$A^{\v}_z$ is a holomorphic endomorphism of the holomorphic vector bundle $(\CC^n, D^{\v}_{\bar z})$}.  In particular, its image
 and kernel form holomorphic subbundles of
$(\CC^n, D^{\v}_{\bar z})$, at least off the discrete subset of $M$ where the rank of $A^{\v}_z$ drops (i.e., does not equal its maximum value), and these subbundles are independent of the local complex coordinate $z$.
By \emph{filling out zeros} as in \cite[Proposition 2.2]{burstall-wood}, these
image and kernel subbundles can be extended to holomorphic subbundles over the whole
of $M$, which we shall denote by
$\image A^{\v}_z$ and $\ker A^{\v}_z$, respectively.

Next, recall that a subbundle $\alpha$ of $\CC^n$ is said to be a \emph{uniton for  a harmonic map } $\v:M\to\U(n)$ if it is
(i) holomorphic with respect to the Koszul--Malgrange holomorphic structure induced by $\v$, i.e.,
$D^{\v}_{\zbar}(\sigma)\in\Gamma(\alpha)\quad(\sigma\in\Gamma(\alpha));$
(ii) closed under the endomorphism $A^{\v}_z$, i.e.,
$A^{\v}_z(\sigma)\in\Gamma(\alpha)\quad(\sigma\in\Gamma(\alpha))$
(here $\Gamma(\cdot)$ denotes the space of smooth sections of a bundle).
Uhlenbeck showed \cite{uhlenbeck} that, \emph{if a subbundle $\alpha\subset\CC^n$ is a uniton for a harmonic map $\v$, then
\begin{equation} \label{uniton}
\wt\v=\v(\pi_{\alpha}-\pi_{\alpha}^{\perp})  \text{ or, equivalently, } -\v(\pi_{\alpha}-\pi_{\alpha}^{\perp}),
\end{equation}
is also harmonic};
$\wt\v$ is said to be obtained from $\v$ by \emph{adding the uniton $\al$}.
Here $\pi_{\alpha}$ (resp.\ $\pi_{\alpha}^{\perp}$) denotes orthogonal projection onto the subbundle $\al$ (resp.\ its orthogonal complement  $\al^{\perp}$ in $\CC^n$).
Uhlenbeck showed further that, \emph{if $\Phi$ is an extended solution associated to $\v$, then $\alpha$ is a uniton for $\v$ if and only if  $\wt\Phi=\Phi(\pi_{\alpha}+\lambda\pi_{\alpha}^{\perp})$ is an extended solution
(associated to $\wt\v=\v(\pi_{\alpha}-\pi_{\alpha}^{\perp})$\,)}.
We shall therefore also say that $\alpha$ is a \emph{uniton for\/ $\Phi$}.

A harmonic map $\v$ is said to be of \emph{finite uniton number} if, for some $r \in \N =\{0,1,2,\ldots \}$,
\begin{equation*} \label{phi-fact}
\v = \v_0(\pi_{\alpha_1}-\pi_{\alpha_1}^{\perp})\cdots(\pi_{\alpha_r}-\pi_{\alpha_r}^{\perp})
\end{equation*}
where $\v_0 \in \U(n)$ is constant and,  for each $i = 1,\ldots, r$, the subbundle $\alpha_i$ is a uniton for the partial
product
$\v_{i-1} = \v_0(\pi_{\alpha_1}-\pi_{\alpha_1}^{\perp})\cdots
	(\pi_{\alpha_{i-1}}-\pi_{\alpha_{i-1}}^{\perp})$.
Equivalently, $\v$ is of finite uniton number if and only if it has an associated extended solution of the form
\begin{equation}\label{Phi-fact2}
\Phi=I (\pi_{\alpha_1}+\lambda\pi_{\alpha_1}^{\perp})\cdots(\pi_{\alpha_r}+\lambda\pi_{\alpha_r}^{\perp}),
\end{equation}
where $I$ denotes the identity matrix and, for each $i$,
$\alpha_i$ is a uniton for
$
\Phi_{i-1} = I(\pi_{\alpha_1}+\lambda\pi_{\alpha_1}^{\perp})\cdots
(\pi_{\alpha_{i-1}}+\lambda\pi_{\alpha_{i-1}}^{\perp}).
$
Any polynomial extended solution $\Phi$ has a factorization \eqref{Phi-fact2} into unitons \cite{uhlenbeck}, thus \emph{a harmonic map $\v$ from a Riemann surface to $\U(n)$ is of finite uniton number if and only if it has a  polynomial associated extended solution}.

\subsection{The criterion for finiteness of the uniton number} \label{subsec:APW criterion}

In \cite[Theorem 3.6]{APW1}, a necessary and sufficient criterion was given for a harmonic map $\v:M \to \U(n)$ to be of finite uniton number.  This may be stated as follows.
Motivated by Segal's Grassmannian model \cite{segal}, see \S \ref{sec:finite type}, we consider the Hilbert space $L^2(S^1,\C^n)$ of all $L^2$ functions $S^1 \ni \la \mapsto f(\la) \in \C^n$ and extend all operators on (sections of) $\CC^n$ to the trivial bundle
$\HH := M \times L^2(S^1,\C^n)$.
Then, for any local coordinate $z$, we define a differential operator $T=T_{\v}$ on sections of  $\HH$  by
\begin{equation} \label{T}
T =T_{\v} = \pa_z + (1-\lambda^{-1})A^{\v}_z = -\lambda^{-1}A^{\v}_z + D^{\v}_z\,.
\end{equation}
 Here $\pa_z$ denotes ordinary (flat) differentiation with respect to  $z$ and $D^{\v}_z$ denotes the derivation $D^{\v}_z = \pa_z + A^{\v}_z$.
For any $i=1,2,\ldots$,
write $T^i = T \circ \cdots \circ T$ ($i$ copies); note that $T^i$ is polynomial in $\la^{-1}$ of degree \emph{at most} $i$.
Then we have from \cite[Theorem 3.6 and Corollary 4.3]{APW1}:

\begin{proposition} \label{S-stabilizes}\emph{[Finiteness criterion]} Let $\v:M \to \U(n)$ be a harmonic map,
Then $\v$ is of finite uniton number if and only if the maximum power of $\la^{-1}$ in  $T_{\v}^i$ stays bounded for $i\in \N$.
\end{proposition}

Also note from \cite[Corollary 4.2]{APW1} that \emph{$\v$ has finite uniton number if and only if its restriction to some open set has finite uniton number}.
In \cite{APW1}, Proposition \ref{S-stabilizes} is called the \emph{bounded powers criterion}.

An immediate consequence of Proposition \ref{S-stabilizes} is that \emph{a harmonic map $\v:M \to \U(n)$ of finite uniton number is \emph{nilconformal}, i.e., has $A^{\v}_z$ nilpotent}.   This follows from holomorphicity of $A^{\v}_z dz$ when $M=S^2$ \cite{uhlenbeck}; for general $M$, it is shown in \cite[Example 4.2]{svensson-wood-twistor}, or follows from the constructions involving nilpotent Lie algebras in \cite{burstall-guest}. In particular the trace of $(A^{\v}_z)^2$ is zero, which says that \emph{a harmonic map of finite uniton number is weakly conformal and so is a minimal branched immersion}, as mentioned in the introduction.

\begin{remark} \label{rem:local}
In the same way \cite{burstall-wood} that $A^{\v}_z dz$ gives a \emph{globally defined} $(1,0)$-form, $T_{\v} dz$ gives a globally defined differential operator.  However, as in \cite{burstall-wood}, for simplicity of notation, we work in a local complex coordinate chart.   All our constructions are independent of that chart.
\end{remark}

\subsection{Maps into a Grassmannian} \label{subsec:Grass}

We now discuss \emph{harmonic maps into Grassmannians} and some geometrical methods for studying them.
For $k \in \{0,1,\ldots, n\}$, let $G_k(\C^n)$ denote the Grassmannian of $k$-dimen\-sional subspaces of $\C^n$. This can be embedded in $\U(n)$ by the
\emph{Cartan embedding} \cite[Proposition 3.42]{cheeger-ebin} which is given, up to left-multiplication by a constant unitary matrix, by $G_k(\C^n) \ni \alpha \mapsto \pi_{\alpha} - \pi_{\alpha}^{\perp} \in \U(n)$.  Since the Cartan embedding is totally geodesic,
by the composition law \cite[\S 5]{eells-sampson}
it preserves harmonicity, i.e.,
\emph{a smooth map $\v$ into $G_k(\C^n)$ is harmonic if and only if its composition with the Cartan embedding is harmonic into $\U(n)$}.

We may consider a smooth map $\v:M \to G_k(\C^n)$ into a Grassmannian as a subbundle, still denoted by $\v$, of the trivial bundle $\CC^n$, with fibre at $z \in M$ given by the $k$-dimensional subspace $\v(z)$.  Note that  $\v$ is harmonic if and only if the map given by the orthogonal complement $\v^{\perp}$ is harmonic; further the two maps have the same Cartan embedding up to sign.

We need some notation from \cite{burstall-wood}:
for two mutually orthogonal subbundles $\v$ and $\psi$ of $\CC^n$, and local complex coordinate $z$ on $M$,
we define the \emph{($\pa'$-)second fundamental form of $\v$ in $\v \oplus \psi$} by
$A'_{\v, \psi}(s) = \pi_{\psi} \circ \pa_z s$ \ ($s \in \Gamma(\v)$); note that this is tensorial. As a special case, we write $A'_{\v} = A'_{\v, \v^{\perp}}$, then a direct calculation from the definitions shows that
$(A^{\v}_z)|_{\v} = -A'_{\v}$ and
$(A^{\v}_z)|_{\v^{\perp}} = -A'_{\v^{\perp}}$.  Similarly, we define the \emph{$\pa''$-second fundamental form of $\v$ in $\v \oplus \psi$} by
$A''_{\v, \psi}(s) = \pi_{\psi} \circ \pa_{\zbar} s$ \ ($s \in \Gamma(\v)$), and write $A''_{\v} = A''_{\v, \v^{\perp}}$, then $(A^{\v}_{\zbar})|_{\v} = -A''_{\v}$ and $(A^{\v}_{\zbar})|_{\v^{\perp}} = -A''_{\v^{\perp}}$.
Note that $A''_{\psi, \v}$ is minus the adjoint of $A'_{\v, \psi}$ and (see \cite[Lemma 1.3]{burstall-wood}, \cite[Theorem 2.1]{chern-wolfson}), \emph{$\v$ is harmonic if and only if $A'_{\v}$ is holomorphic, and this holds if and only if $A''_{\v}$ is antiholomorphic}.

Let $\v:M \to G_k(\C^n)$ be a smooth map and $\alpha$
a subbundle of $\CC^n$.  Then \cite{uhlenbeck},
$\wt{\v} = -\v(\pi_{\alpha}-\pi_{\alpha}^{\perp})$ has image in a Grassmannian if and only if $\pi_{\alpha}$ commutes with $\pi_{\v}$, and this holds if and only if $\alpha$ is the direct sum of subbundles $\beta$ of $\v$ and $\gamma$ of $\v^{\perp}$, in which case
$\wt{\v} = (\v \ominus \beta) \oplus \gamma$  --- for this formula, we need to choose the minus sign in \eqref{uniton}; note that $\pm\wt{\v}$ are orthogonal complements.
Then \emph{$\alpha$ is a uniton for $\v$ if and only if $\be$ and $\ga$ are holomorphic subbundles of $\v$ and $\v^{\perp}$, respectively, with $A'_{\v}(\be) \subseteq \ga$ and $A'_{\v^{\perp}}(\ga) \subseteq \be$}.

An important example is when we choose $\be \subseteq \ker (A'_{\v^{\perp}} \!\circ\! A'_{\v})$ and set $\ga = A'_{\v}(\be) := \image(A'_{\v}|_{\be})$. Since $A'_{\v}$ and $\be$ are holomorphic, we may complete this image to a subbundle by filling out zeros thus obtaining a harmonic map $\wt{\v} = (\v \ominus \beta) \oplus A'_{\v}(\be)$.
Following \cite[\S 2B]{burstall-wood}, we say that $\wt\v$ is obtained from $\v$ by \emph{forward replacement of $\be$}.  Dually, let $\ga$ be an antiholomorphic subbundle of $\v$ such that $\ga \subseteq \ker (A''_{\v^{\perp}} \circ A''_{\v})$.
Then $\wt\v =(\v \ominus \ga) \oplus A''_{\v}(\ga)$ is harmonic;
we say that
$\wt\v$ is obtained from $\v$ by \emph{backward replacement of $\ga$}.  Note that \emph{reversing the orientation of $M$}, i.e., replacing its complex structure by the conjugate complex structure, turns forward replacements into backward replacements and vice-versa.

A special case of forward replacement is when we choose $\beta = \v$ and $\gamma = \image A'_{\v}$, in which case
$\wt{\v} = \image A'_{\v}$.  We write
$G'(\v):= \image A'_{\v}$; this is called the
\emph{($\pa'$)-Gauss bundle} of $\v$ \cite[\S 2B]{burstall-wood} or
 $\pa$-transform \cite[Theorem 2.2]{chern-wolfson}.
Similarly, we define the
\emph{$\pa''$-Gauss bundle} of $\v$ by $G''(\v) := \image A''_{\v}$.

\section{Diagrams, cycles, and finite uniton number} \label{diagrams and cycles}

\subsection{Diagrams} \label{diagrams}
In \cite[\S 4]{APW1}  the present authors and A.~Aleman gave a  test for finiteness of the uniton number of a harmonic map into a Grassmannian based on the existence of certain cycles in a diagram associated to the harmonic map.
Here we extend that work to give more powerful tests.

To display second fundamental forms we shall use \emph{(harmonic) diagrams}, in the sense of \cite{burstall-wood}, i.e., directed graphs whose vertices (nodes) are mutually orthogonal subbundles $\psi_i$ with sum $\CC^n$.
For each ordered pair $(i,j)$ with $i \neq j$, the arrow (i.e., directed edge) from $\psi_i$ to $\psi_j$ represents the second fundamental form $A'_{\psi_i,\psi_j}$.   We shall call an arrow with $i \neq j$ \emph{proper}; the absence of that arrow indicates that $A'_{\psi_i,\psi_j}$ is known to vanish. 
In this paper, we will also associate a \emph{self-arrow} to each vertex $\psi_i$ --- that is, an arrow whose endpoints are the same vertex $\psi_i$ --- which represents the derivation
$D^{\psi_i}_z=\pi_{\psi_i}\circ \pa_z|_{\psi_i}$. However, self-arrows will not be shown  in the displayed diagrams below.

The most basic diagram for a smooth map $\v:M \to G_k(\C^n)$ is

\vspace{-3ex}
\begin{equation}
\begin{gathered}\label{diag:basic}
%\xymatrixrowsep{1.5pc}
\xymatrixcolsep{3pc}
\xymatrix{
	\v \ar[r]_(0.5){A'_{\v}} & \v^{\perp}
\ar@/_0.7pc/[l]_{A'_{\v^{\perp}}}
}
\end{gathered} \vspace{-0.8ex} \end{equation}
As in \S\S \ref{subsec:background}, \ref{subsec:Grass}, the second fundamental forms $A'_{\v}$ and $A'_{\v^{\perp}}$ are holomorphic if and only if $\v$ is harmonic.

We say that a diagram $\Delta_2$ is a \emph{refinement} of a diagram $\Delta_1$ if all the vertices of diagram $\Delta_1$ are (orthogonal direct) sums of vertices of $\Delta_2$.   By a diagram \emph{associated to} a smooth map $\v:M \to G_k(\C^n)$ we mean one which is a refinement of the basic diagram \eqref{diag:basic}, i.e., each vertex represents a subbundle of $\v$ or $\v^{\perp}$.

For each path $C$ in a diagram from a vertex  $\be$ to a vertex $\ga$ and on each coordinate patch there is a \emph{corresponding operator  $\Cc:\Gamma(\be) \to \Gamma(\ga)$} which we denote by the same letter in boldface, given by the composition of the second fundamental forms and derivations represented by the edges of the path. If the path contains no self-arrows, then $\Cc$ is a \emph{bundle map} $\be \to \ga$.
By a \emph{cycle} on a vertex $\be$
we mean a path which starts and finishes at $\be$; note that  there may be self-arrows, and vertices may occur more than once.

In \cite[Proposition 1.5]{burstall-wood}, there is given a simple sufficient condition that a second fundamental form in a diagram be holomorphic --- note that Proposition 1.6 in that paper contains an error, see \cite{burstall-wood-correction}.
This leads to the idea of \emph{holomorphic cycles} (or \emph{circuits}) where each arrow in the cycle is holomorphic.  The key tool in that paper \cite[Proposition 1.8]{burstall-wood}, which appeared in \cite[\S2]{ramanathan}, is that a holomorphic cycle of length $\ell$ on a vertex $\al$ gives a globally defined section of $\otimes^{\ell} T^*M \otimes L(\al,\al)$, and if $M=S^2$, this must be \emph{nilpotent}; indeed any such section is given locally by $dz^\ell\otimes A$, and the coefficients of the characteristic polynomial of $A$ produce globally defined holomorphic differentials on $S^2$ which must vanish  \cite[Proposition 1.8]{burstall-wood}.
This leads to vanishing theorems giving methods to simplify, and so understand, harmonic maps from $S^2$.  For other Riemann surfaces, this argument does not work in general; instead we use the following new idea coming from the development in \cite{APW1} which \emph{does} work for all Riemann surfaces.

Call an arrow \emph{external} if one of its vertices is in $\v$ and the other is in $\v^\perp$, otherwise call it \emph{internal}.
A path is \emph{external} if it contains at least one external arrow.
Given any arrow $\psi_i \to \psi_j$, we have a corresponding component $\pi_{\psi_j} \circ T|_{\psi_i}$ of the operator $T$ defined by \eqref{T}.
If $\psi_i \to \psi_j$ is external, this component is given by
$\pi_{\psi_j} \circ T|_{\psi_i} = -\la^{-1}\pi_{\psi_j} \circ A^{\v}_z|_{\psi_i} = \la^{-1} A'_{\psi_i,\psi_j}
= \la^{-1}\pi_{\psi_j} \circ \pa_z|_{\psi_i}$,
whereas, if it is internal with $i \neq j$, it is $\pi_{\psi_j} \circ T|_{\psi_i}  = \pi_{\psi_j} \circ D^{\v}_z|_{\psi_i}
= A'_{\psi_i,\psi_j} = \pi_{\psi_j} \circ \pa_z|_{\psi_i}$.  Finally, if $i=j$, the component
$\pi_{\psi_i} \circ T|_{\psi_i}$ is
$\pi_{\psi_i} \circ D^{\v}_z|_{\psi_i} = D^{\psi_i}_z =
 \pi_{\psi_i} \circ \pa_z|_{\psi_i}$.

Given a diagram associated to a smooth map $\v:M \to G_k(\C^n)$, we say that a path $C$ has  \emph{degree} $m$ if it has $m$ external arrows and \emph{type} $(\ell,m)$ if it has length $\ell$ and degree $m$.   Then the composition of the components of $T$ corresponding to each arrow along the path gives a term $\la^{-m}\Cc$ in $T^{\ell}$;
note that this could be zero.  Let $\al$ be a subbundle of the initial vertex $\be$ of $C$.  We say that $C$ is \emph{zero} (resp. \emph{zero on $\al$})
and we write $C=0$ (resp. $C|_{\al} = 0$),
if the corresponding operator $\Cc$ is zero on (sections of) $\be$ (resp. $\al$);
similarly, we say that $C$ is \emph{nilpotent} if $\Cc$ is.
 Recall that an arrow which is not a self-arrow is called \emph{proper}: such arrows represent second fundamental forms.

\begin{lemma}\label{lem:bmap}
   Consider a harmonic diagram $\Delta$ and let $C$ be a path in $\Delta$ of length $\ell$ from a vertex $\be$ to a vertex $\ga$;
  if $\be=\ga$, suppose further that $C$ is not of the form $S^\ell$, where  $S$ is the self-arrow on $\be=\ga$.
 Let $\al$ be a subbundle of $\be$.  If any  path of length $\ell$  from $\be$ to $\ga$  containing exactly the same  proper arrows as $C$ is either zero  on $\al$ or equal to $C$, then  $\Cc$  restricts to a bundle map from $\al$ to $\ga$.
\end{lemma}
\begin{proof}
  We write
$\Cc = \mathbf{A}_\ell\circ \mathbf{A}_{\ell-1}\circ\ldots\circ \mathbf{A}_1,$
where $\mathbf{A}_i$ is the operator  corresponding to the $i$th arrow in $C$.
If  $\Cc$ is zero on $\al$, then
 $\Cc|_{\al}$ is obviously a bundle map.  So now assume that  $\Cc|_{\al}$ is non-zero.
If $C$ does not contain any self-arrow,  then each $\mathbf{A}_i$ is a bundle map and, consequently, $\Cc|_{\al}$  is a bundle map.

Suppose that $C$ contains exactly $k\geq 0$ self-arrows.
Given a smooth complex function $f$ on $M$ and a section $s\in  \Gamma( \al)$, we claim that
\begin{equation}\label{bmap}
 \Cc(fs)=f\Cc_0(s)+f'\Cc_1(s)+\ldots +f^{(k)}\Cc_k(s)
 \end{equation}
where $\Cc_0=\Cc$ and, for each $j=1,\ldots,k$, $\Cc_j$ is the sum of all  operators $\BB_{mj}$, with $m=1,\ldots,\binom{k}{j}$,  corresponding to paths from $\be$ to $\ga$ obtained from $C$ by removing exactly $j$ self-arrows. We prove this claim by induction  on the length $\ell$ of $\Cc$. For $\ell=1$, we have $k=0$ or $k=1$, and  formula \eqref{bmap} obviously holds in both cases.
 For the induction step, suppose that the formula holds for any path of length $\ell-1$.  Let $\widehat{C}$ be the path in $\Delta$ of length $\ell-1$ obtained from  $C$ by removing the last arrow,  thus $\widehat\Cc=  \mathbf{A}_{\ell-1}\circ\ldots\circ \mathbf{A}_1.$ Assume first that the last arrow is not a self-arrow, so that $\mathbf{A}_\ell$ is a bundle map. By the induction hypothesis, $\widehat\Cc(fs)=f\widehat\Cc_0(s)+f'\widehat\Cc_1(s)+\ldots +f^{(k)}\widehat\Cc_k(s)$; since $\mathbf{A}_\ell$ is a bundle map, we have  $\Cc_i=\mathbf{A}_\ell\circ \widehat\Cc_i$; hence we see that \eqref{bmap} holds for $\Cc= \mathbf{A}_\ell\circ  \widehat\Cc$. Suppose now that $\mathbf{A}_\ell$ corresponds to a self-arrow. In this case,  $\widehat{C}$ has $k-1$ self-arrows and $C$ has $k$ self-arrows. By the induction hypothesis,
$\widehat\Cc(fs)=f\widehat\Cc_0(s)+f'\widehat\Cc_1(s)+\ldots +f^{(k-1)}\widehat\Cc_{k-1}(s)$. We clearly have $\Cc_0=\mathbf{A}_\ell\circ\widehat\Cc_0$, $\Cc_k=\widehat\Cc_{k-1}$,  and $\Cc_i=\widehat\Cc_{i-1}+\mathbf{A}_\ell\circ\widehat\Cc_{i}$ for $i=1,\ldots,k-1$; then, since  $\mathbf{A}_\ell$ is a derivation, we obtain:
\begin{align*}
  \Cc(fs)&=\mathbf{A}_\ell\circ \widehat\Cc(fs)\\
  &=\mathbf{A}_\ell\big(f\widehat\Cc_0(s)+f'\widehat\Cc_1(s)+\ldots +f^{(k-1)}\widehat\Cc_{k-1}(s)  \big)\\
  &= f\mathbf{A}_\ell\circ\widehat\Cc_0(s)+f'(\widehat\Cc_0(s)+ \mathbf{A}_\ell\circ\widehat\Cc_1(s))+\ldots +f^{(k)}\widehat\Cc_{k-1}(s)\\
  &= f\Cc_0(s)+f'\Cc_1(s)+\ldots +f^{(k)}\Cc_k(s),
\end{align*}
as claimed.

Now we prove the lemma by contradiction. Suppose that $\Cc|_{\al}$ is not a bundle map.  Then, by \eqref{bmap},
we can choose a section  $s\in  \Gamma(\al)$ and $j \in \{1,\ldots,k\}$ such that $\Cc_j(s)\neq 0$,  and  consequently $\BB_{mj}(s)\neq 0$ for some $m \in \{1,\ldots,\binom{k}{j}\}$.
Let $j_0 $ be the  largest $j$ such that $\BB_{mj}(s)\neq 0$ for some $m$, and choose $m_0$ such that $\BB_{m_0j_0}(s)\neq 0$. Recall that $\BB_{m_0j_0}$ corresponds to a path
 $B_{m_0j_0}$ from $\be$ to $\ga$ obtained from $C$ by removing some set of $j_0$ self-arrows, hence it has length $\ell-j_0$.
The idea  is to construct a path $\widetilde{C}\neq C$ of length $\ell$ from $\beta$ to $\ga$, containing exactly the same  proper arrows as $C$, with  $\tilde\Cc|_{\al}$ non-zero by adding $j_0$ self-arrows (on $\beta$ or $\ga$) to $B_{m_0j_0}$.

Set $B=B_{m_0,j_0}$ and   consider the path $\widetilde{C}=S^{j_0}\circ B$ from $\beta$ to $\ga$,  which  has length $\ell$,  where $S$ is the self-arrow on $\ga$.
Since $B$ has $k-j_0$ self-arrows, we can write $\BB(fs)=f\BB_0(s)+f'\BB_1(s)+\ldots +f^{(k-j_0)}\BB_{k-j_0}(s)$, for the section $s \in \Gamma(\al)$ chosen above, and any smooth complex function $f$, as in \eqref{bmap}.
If $\BB_i(s)\neq 0$ for some $i>0$, then there exists a path from $\beta$ to $\ga$ obtained from $B$ by removing a set of $i$ self-arrows giving an operator  which is non-zero on $\al$; but then such  a  path is obtained from $C$ by removing a set of $j_0+i$ self-arrows, i.e., $ \BB_{m,j_0+i}(s)\neq 0$ for some $m$, which contradicts our definition of $j_0$.
Hence  $\BB(fs)=f\mathbf B(s)$ and $\mathbf B(s)\neq 0$.
Since $S$ is a self-arrow, $\mathbf S$ is a derivation, and this implies that
$$\widetilde{\mathbf C}(fs)=\mathbf S^{j_0}\circ \mathbf B(fs)=\mathbf S^{j_0}\big(f\mathbf B(s)\big)=\sum_{i=0}^{j_0}\binom{j_0}{i}f^{(j_0-i)}\mathbf{S}^i(\BB(s)).$$
Since this holds for any smooth function $f$ and $\BB(s)\neq 0$, the operator $\widetilde{\mathbf C}$ is non-zero on $\al$.
Hence,  if $\widetilde{C}\neq C$,  this  contradicts the assumption that $C$ is the unique path
from $\be$ to $\ga$ of length $\ell$ which is non-zero on $\al$.  If $\widetilde{C}= C$,  we consider instead the path $\widetilde{C}= B\circ S^{j_0}$,  which is also of type $\ell$ and,  by a similar argument,
 is non-zero  on $\al$. Since $C$ has at least one arrow which is not a self-arrow, if $C=S^{j_0}\circ B$ then $C\neq B\circ S^{j_0}$, and we are done.
\end{proof}

We remark that, although it restricts to a bundle map, $\Cc$ may still contain self-mappings, cf.\ the second approach in Proposition \ref{Frenet-mixed}.
We now give an important consequence of the finiteness criterion (Proposition \ref{S-stabilizes}),  which we state in a general setting.

\begin{proposition} \label{nilpotency test}
Let $\v:M \to G_k(\C^n)$ be a harmonic map of finite uniton number, and consider an associated diagram  $\Delta$. Let $\psi$ be a vertex of $\Delta$ with $\psi \subseteq \v$ and  let $\al$ be a subbundle of $\psi$ (which may not be a vertex of $\Delta$).
Let $C$ be an external cycle of type $(\ell,m)$ on  $\psi$ whose corresponding operator $\Cc$ sends (sections of) $\al$ to $\al$.
Then  $\Cc$ restricts to a nilpotent bundle map
$\Cc|_{\al}:\al \to \al$ if, for each $j \in \N$,
any cycle on  $\psi$ of type $(j\ell,jm)$  is zero on $\al$ or is equal to $C^j$.
\end{proposition}

\begin{proof}
Since $C$ is external, $C$ contains at least one arrow which is not a self-arrow. Then,
by Lemma \ref{lem:bmap}, the operator $\Cc|_\alpha$ is a bundle map.

Suppose that  $\Cc|_\alpha$ is not nilpotent.  Then,  for each $j$, $\Cc^j(\al) \neq 0$.
Since $C^j$ is the unique cycle on  $\psi$ of type $(j\ell,jm)$  which is non-zero on $\al$, it follows that $\pi_{\la^{-jm}\al}T^{j\ell}(\al) \neq 0$, thus, for each $j$, $T^{j\ell}$ has a non-vanishing term in $\la^{-jm}$, contradicting Proposition \ref{S-stabilizes}.
\end{proof}

\begin{remark} \label{rem:nilpotency test}
\begin{itemize}
\item[(i)]  If $\al$ is a bundle of rank one, a bundle map $\al \to \al$  is nilpotent if and only if it is zero.

 \item[(ii)] In our main applications of Proposition \ref{nilpotency test}, namely Propositions \ref{ce}, \ref{ce^s}, \ref{cece^2}, \ref{nilorder p} and \ref{c^2ece}, we consider a diagram $\Delta$ which has $\v$ as a vertex, set $\psi = \v$, and take $\al$ to be a proper subbundle of $\v$ which is not a vertex of $\Delta$.
\item[(iii)] For other applications we take a diagram where $\al$
\emph{is} a vertex and we set $\psi = \al$. See Proposition \ref{Frenet-mixed} for both approaches.
\end{itemize}
\end{remark}

To implement Proposition \ref{nilpotency test}
with $\psi=\v$, the following is often useful.

\begin{lemma} \label{uniqueness test} Let $\v:M \to G_k(\C^n)$ be a smooth map and consider an associated diagram which has $\v$ as a vertex. Let $\al$ be a  subbundle of $\v$.
Let $C$ be a cycle on $\v$ of type $(\ell,m)$  whose corresponding operator $\Cc$ sends (sections of) $\al$ to $\al$.
Suppose that {\rm (i)} any non-zero  cycle on $\v$ of degree $m$ has length at least $\ell;$
{\rm (ii)}  any cycle on $\v$ of type
$(\ell,m)$ is zero on $\al$ or is equal to $C$.   Then, for each $j \in \N$,
 any cycle on $\v$ of type $(j\ell,jm)$ is zero on $\al$ or is equal to $C^j$.
\end{lemma}

\begin{proof}
Since $C$ starts and finishes  on $\v$, $m$ must be even.
Let $B$ be a non-zero cycle on $\v$ of type $(j\ell,jm)$; it is the composition of $j$ non-zero cycles $C_i$  on $\v$ of degree $m$.
By hypothesis (i), each $C_i$ has length \emph{at least} $\ell$; since the total length of $B$ is $j\ell$, each $C_i$ has length \emph{exactly} $\ell$.
Finally, by induction on $i$ and hypothesis (ii), each $C_i$ is zero on $\al$ or equal to $C$, so that  $B$ is zero on $\al$ or equal to $C^j$.
\end{proof}

Note that the lemma still holds when $C=0$,  in which case the conclusion is that there are no non-zero cycles  on $\v$ of type $(j\ell,jm)$ for any $j$.

As a first example,  let $\v:M \to G_k(\C^n)$ be a harmonic map and let $C$ be the cycle $\v \to \v^{\perp} \to \v$ on $\v$ of type $(2,2)$ in the diagram \eqref{diag:basic};  the corresponding  operator is the bundle map $\Cc :\v \to \v$ given by $\Cc = A'_{\v^{\perp}} \circ A'_{\v} = (A^{\v}_z)^2|_{\v}$. The conditions (i) and (ii) of Lemma \ref{uniqueness test} are satisfied with $\al = \v$. Hence, by Proposition \ref{nilpotency test}
(with $\psi=\v$), \emph{if $\v$ is of finite uniton number, $\Cc$ is nilpotent}; it follows that $A^{\v}_z$ is nilpotent, as was already noted above.

A harmonic map $\v:M \to G_k(\C^n)$ is called \emph{$\pa'$-irreducible} (resp.\
\emph{$\pa''$-irreducible}) if $\rk G'(\v) = \rk \v$
(resp. $\rk G''(\v) = \rk \v$),
equivalently, $A'_{\v}$ (resp.\ $A''_{\v}$) has maximal rank off an isolated set of points; the structure of \emph{reducible}  harmonic maps is given by the Reduction Theorem \cite[Theorem 4.1]{burstall-wood}.
Nilpotency of $A^{\v}_z$ implies that either $\v$ or $\v^{\perp}$ is reducible.  In particular, \emph{the description of J.~Ramanathan \cite{ramanathan} of harmonic maps from $S^2$ to $G_2(\C^4)$ as stated in \cite[\S 5D]{burstall-wood} immediately generalizes to
harmonic maps of finite uniton number from any Riemann surface}.

\subsection{More nilpotent cycles} \label{subsec:more}

For a \emph{harmonic} map $\v$, we recalled above that the \emph{($\pa'$-)Gauss bundle} $G'(\v)$ is the image of $A'_{\v}$ (completed to a subbundle by filling out zeros); this bundle is harmonic.  We iterate this construction to give the \emph{$i$th ($\pa'$-)Gauss bundle} $G^{(i)}(\v)$ for $i=1,2,\ldots$.
by setting $G^{(0)}(\v) = \v$, $G^{(1)}(\v) = G'(\v)$, $G^{(i+1)}(\v) = G'(G^{(i)}(\v))$;
see \cite{chern-wolfson} for a moving frames approach.

Similarly, we define the \emph{$\pa''$-Gauss bundle} of $\v$ by
$G''(\v) = \image(A^{\v}_{\zbar}|_{\v})$, and iterate this to obtain $G^{(-1)}(\v) = G''(\v)$, $G^{(-i-1)}(\v) = G''(G^{(-i)}(\v))$ \ $(i =0,1,\ldots)$, so that $G^{(i)}(\v)$ is defined for all $i \in \Z$.   Note that $G''(G'(\v)) \subseteq \v$ (resp. $G'(G''(\v)) \subseteq \v$) with equality if and only if $\v$ is $\pa'$-irreducible (resp. $\pa''$-irreducible) \cite[Proposition 2.3]{burstall-wood}.  The sequence of Gauss bundles $G^{(i)}(\v)$ \ $(i \in \Z)$ is called the
\emph{harmonic sequence of $\v$} \cite{chern-wolfson,wolfson}.

The \emph{(complex) isotropy order} of a harmonic map $\v:M \to G_{m}(\C^n)$ into a (complex) Grassmannian is defined to be the greatest value of $r \in \{1,2,\ldots,\infty\}$ such that $\v$ is orthogonal to $G^{(i)}(\v)$ for all $i$ with $1 \leq i \leq r$. Equivalently \cite[Lemma 3.1]{burstall-wood}, the isotropy order is the greatest value of $r$ such that $G^{(i)}(\v)$ and $G^{(j)}(\v)$ are orthogonal for all $i,j \in \Z$ with $i \neq j$ and $1 \leq |i-j| \leq r$.   A harmonic map with infinite isotropy order is called  \emph{(strongly) isotropic}.

For a harmonic map $\v:M \to G_k(\C^n)$ of finite isotropy order $r$,  we define the  \emph{first return (path)} to be the path
$c = c_r = c(\v)  = c_r(\v): \v \to G'(\v) \to \cdots \to G^{(r-1)}(\v) \to \wt{R} \to \v$, i.e. the cycle on $\v$ shown in the following diagram:

\vspace{0ex}
\begin{equation}%BED-W 2.2
\begin{gathered}\label{diag:first return0}
%\xymatrixrowsep{1.5pc}
\xymatrixcolsep{3pc}
\xymatrix{
	\v \ar[r]_(0.4){A'_{\v}} & \ G'(\v)
	\ar[r]_{A'_{G^{(1)}(\v)}} & \space\space\space\cdots &
\ar[r]_(0.4){A'_{G^{(r-2)}(\v)}} & G^{(r-1)}(\v)\ar[r]_(0.7){ A'_{G^{(r-1)}(\v)}} & \wt{R}
\ar@/_1.2pc/[lllll]
}
\end{gathered} \end{equation}
where $\wt{R}:= \bigl(\sum_{i=0}^{r-1} G^{(i)}(\v)\bigr)^{\perp}$ and the curved arrow represents $A'_{\wt{R}, \v}$.

 The \emph{first ($\pa'$)-return map}  $\cc = \cc_r = \cc(\v)  = \cc_r(\v)$ is the corresponding bundle map $\v \to \v$
(see \cite{bahy-wood-G2} where it is denoted by $c'_r(\v)$).  Thus, the first return map is the composition of second fundamental forms:
\begin{eqnarray} \label{first returnc}
	\cc &=& A'_{G^{(r)}(\v), \v} \circ
	A'_{G^{(r-1)}(\v)} \circ \cdots \circ A'_{G^{(1)}(\v)} \circ A'_{\v}\\
	&=& \pi_{\v} \circ A'_{G^{(r)}(\v)} \circ
	A'_{G^{(r-1)}(\v)} \circ \cdots \circ A'_{G^{(1)}(\v)} \circ A'_{\v}.\nonumber
\end{eqnarray}

Note that the maps $A'_{\v}$, $A'_{G^{(1)}(\v)}$, $\ldots$, $A'_{G^{(r-2)}(\v)}$ in \eqref{diag:first return0} are surjective and
$A'_{G^{(r-1)}(\v)}$ has image $G^{(r)}(\v)$, which lies in $\wt{R}$.

Using the test for holomorphicity in \cite[Proposition 1.5]{burstall-wood}, the first return map $\cc:\v \to \v$ is a holomorphic cycle and so is nilpotent if $M = S^2$.  For an \emph{arbitrary} Riemann surface $M$, it is shown in \cite[Theorem 4.10]{APW1}, that \emph{$\cc$ is still nilpotent if $\v$ has finite uniton number}. Note that this also follows from Proposition \ref{nilpotency test}, or Lemma \ref{uniqueness test}, applied to diagram  \eqref{diag:first return0} with $\al = \v$ and $(\ell,m) = (r+1,2)$, since $c$ clearly satisfies the required conditions; furthermore, nilpotency of $\cc$ generalizes that of $A^{\v}_z$ mentioned above.

 This immediately tells us that some theorems which depend only on this fact extend.  For example, the theorem of A.~R.~Aithal \cite{aithal-HP2} extends to give a description of harmonic maps of finite uniton number from any Riemann surface to $\HP^2$ (see also \cite[\S 6(B)]{bahy-wood-HPn}); a generalization of this to harmonic maps of finite uniton number from a Riemann surface into $\HP^n$ was given in \cite[\S 5]{Pacheco-Sp(n)} by using loop group methods.
 
In diagram \eqref{diag:first return0},  the first return path is the only external cycle on $\v$ of length $r+1$ and any other  external cycle on $\v$ is longer.
Next we consider the following refinement of diagram \eqref{diag:first return0}:

\vspace{1ex}
\begin{equation}
\begin{gathered}\label{diag:second return}
%\xymatrixrowsep{1.5pc}
\xymatrixcolsep{1.9pc}
\xymatrix{
	\v \ar[r]_(0.4){A'_{\v}} & \ G'(\v)
	\ar[r]_(0.6){A'_{G^{(1)}(\v)}} & \space\space\space\cdots &
\ar[r]_(0.4){A'_{G^{(r-2)}(\v)}} & G^{(r-1)}(\v)\ar[r]_(0.6){ A'_{G^{(r-1)}(\v)}} & G^{(r)}(\v)\ar@/_1.1pc/[lllll]
\ar[r]_(0.7){A'_{G^{(r)}(\v),R}} & R
\ar@/_1.8pc/[llllll]
}
\end{gathered} \end{equation}
where $R:= \bigl(\sum_{i=0}^{r} G^{(i)}(\v)\bigr)^{\perp}$,
and the outer and inner curved arrows represents $A'_{R, \v}$ and
$A'_{G^{(r)}(\v), \v}$, respectively.

By \eqref{first returnc},   the first return map $\mathbf{c}$ corresponds to the \emph{inner} cycle $\v \to G'(\v) \to \cdots \to G^{(r)}(\v) \to \v$;  as in diagram \eqref{diag:first return0} we call this the \emph{first return path} and denote it by $c$. We define the  \emph{second return path} $e=e(\v)$ to be the \emph{outer} cycle
$\v \to G'(\v) \to \cdots \to G^{(r)}(\v) \to R \to \v$ of \eqref{diag:second return}, and the \emph{second return map} $\ee = \ee(\v):\v \to \v$ to be the  corresponding bundle map, i.e., the composition of the second fundamental forms:
\begin{equation*} \label{def:second return}
	\ee = \ee(\v) = A'_{R, \v} \circ A'_{G^{(r)}(\v), R} \circ
	A'_{G^{(r-1)}(\v)} \circ \cdots \circ A'_{G^{(1)}(\v)} \circ A'_{\v}.
\end{equation*}

We can describe cycles on $\v$ of  type $(r+2,2)$ as follows. It is convenient to use the symbol $\circ$ to denote composition of paths as well as operators, for example, $e \circ c$ means the path $c$ followed by the path $e$; note that the corresponding operator is the composition of operators $\ee \circ \cc$.

\begin{lemma} \label{epsilon}
 Let $\v:M \to G_k(\C^n)$ be a harmonic map of finite isotropy order $r$.  Consider the diagram \eqref{diag:second return}.
Then any  cycle $\ep$ on $\v$ of type $(r+2,2)$ is one of the following$:$

\centerline{
{\rm (1)} $e$, \space\space {\rm (2)} $u \circ c$, \space\space
{\rm (3)} $c \circ u$, \space\space  {\rm (4)} $\widehat c$.}
\noindent Here
$u$ denotes the self-arrow on $\v$ and $\widehat c$ is obtained from
$c$ by  inserting a self-arrow on $G^{(i)}(\v)$
for some $i \in \{1,\ldots, r\}$, in particular,
$\image{\widehat \cc} \subseteq \image{\cc}$.
\end{lemma}

\begin{proof}
If $\ep$ includes $R$ it must be $e$, otherwise it would be too long.  If not, it includes $G^{(r)}(\v) \to \v$ and must be $c$ with a self-arrow on $G^{(i)}(\v)$ inserted for some $i \in \{0,1,\ldots, r\}$ giving (2), (3) or (4).

Since the maps $A'_{\v}$, $A'_{G^{(1)}(\v)}, \ldots, A'_{G^{(r-1)}(\v)}$ in \eqref{diag:second return} are surjective, it is clear that  $\image{\widehat \cc} \subseteq \image{\cc}$.
\end{proof}

\begin{example} \label{ex:c2=0}
Let $\v:M\to G_k(\C^n)$ be a harmonic map of finite isotropy order $r$ with $\cc^2=0$.  Let $\al$ be any subbundle of $\v$ with $\image \cc  \subseteq \al \subseteq \ker \cc $ and set
$\be = \v \ominus \al$. 
 Define holomorphic subbundles $\al_i$ of $G^{(i)}(\v)$ inductively by
$\al_0 = \al$, $\al_{i} = \image(A'_{G^{(i-1)}(\v)}|_{\al_{i-1}})$  for $i = 1,\ldots,r$, and set
$\be_i = G^{(i)}(\v) \ominus \al_i$, 
$R = \bigl(\sum_{i=0}^{r} G^{(i)}(\v)\bigr)^{\perp}$;
 note that $R$ or some of the $\al_i$ may be zero, but they
 are non-zero if  all $G^{(i)}(\v)$ \ $(i=0,1,\ldots, r)$ are irreducible.

 Then we have the following refinement of diagram \eqref{diag:second return}\/$:$

\vspace{2ex}
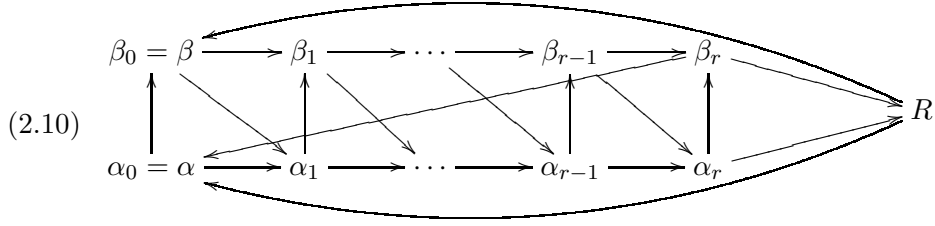
\begin{equation}
\begin{gathered}\label{diag:second return2}
\xymatrixrowsep{0.6pc}
\xymatrixcolsep{2.4pc}
\xymatrix{
\be_0 = \be \ar[r]\ar[rdd] & \be_1 \ar[r]\ar[rdd] & \space\space\space\cdots \ar[rdd]\ar[r] &\be_{r-1} \ar[r]\ar[rdd]
	 &\be_r\ar[lllldd]\ar[rrd]
\\
{} & {} &{} &{} &{} &{} &R\ar@/^2.4pc/[lllllld]\ar@/_2.4pc/[llllllu]
\\
\al_0 = \al \ar[r]\ar[uu] & \al_1 \ar[r]\ar[uu] & \space\space\space			\cdots \ar[r] &\al_{r-1} \ar[r]\ar[uu] & \al_r \ar[uu]\ar[rru]
}
\vspace{2.5ex}
\end{gathered} \end{equation}

Let $\ep$ be one of the cycles of Lemma \ref{epsilon}. Then the operator $\epp|_\al$ coincides with the operator corresponding to one of the following paths$:$

{\rm (1)}$:$ $\al \to \al_1 \to \cdots \to \al_r \to R \to \v \,$ (if $\ep=e$);

{\rm (2)}$:$ zero cycle (if $\ep=u\circ c$)$;$

{\rm (3)}$:$ $\al \to \be \to \be_1 \to \cdots \to \be_r \to \al\,$ (if $\ep=c\circ u$)$;$

{\rm (4)}$:$ $\al \to \al_1 \to \cdots \al_i \to \be_i \to \cdots \to \be_r \to \al\,$ (if $\ep=\widehat{c} $).

In (2), with $\ep=u\circ c$, then $\epp|_\al=0$ since $\al \subseteq \ker \cc$.
In (3) and (4) we have included an  arrow $\al_i \to \be_i$; the corresponding operator $\al_i \to \be_i$ is, of course, the second fundamental form $A'_{\al_i,\be_i}$.  Comparing with
diagram \eqref{diag:second return}, this is the {$\be_i$-}component of the restriction to $\al_i$ of the operator
$D^{G^{(i)}(\v)}_z$ corresponding to the self-arrow on $G^{(i)}(\v)$. If we had included a self-arrow on $\alpha_i$ instead of the arrow $\al_i \to \be_i$, the corresponding  operator would be the other component of  $D^{G^{(i)}(\v)}_z|_{\al_i}$, namely $\pi_{\al_i} \circ D^{G^{(i)}(\v)}_z:\Gamma(\al_i) \to \Gamma(\al_i)$, and the operator $\epp$ would be zero on $\al$ since $\al \subseteq \ker \cc$.
\end{example}

Since there are four possibilities for $w$ in Lemma \ref{epsilon}, $e$ does not satisfy the conditions of Proposition \ref{nilpotency test} and does not, in general, give a nilpotent cycle on $\v$ or $\al$; however when combined with $c$ we can construct nilpotent cycles from $e$ as we shall see below.

Define the \emph{nilorder} of the first return path $c$ to be  the nilorder of its corresponding operator $\cc:\v \to \v$, i.e., the least value of $p \in \{1,2,\ldots, \infty\}$ such that $\cc^p=0$; as above this is finite for a harmonic map of finite uniton number.
 We now discuss how to modify a harmonic map of finite uniton number and finite isotropy order into a Grassmannian by increasing its isotropy order until it becomes \emph{reducible},
 leading to a description of the harmonic map.  The modifications are given by replacement of a subbundle found by finding a sequence of new nilpotent cycles.  There are three different cases, which we discuss in the next three subsections.

\subsection{First return map of nilorder $2$}\label{subsec:nilorder 2}

\begin{proposition} \label{ce}
Let $\v:M \to G_k(\C^n)$ be a harmonic map of finite uniton number and finite isotropy order, with first return  map $\cc$ of nilorder $2$.
Suppose that $\al$ is any subbundle of $\v$ with $\image \cc  \subseteq \al \subseteq \ker \cc $.
Let $\DD_1$ be the bundle map given
 by $\cc \circ \ee|_{\al}:\al \to \al$\,. Then $\DD_1$ is nilpotent.
\end{proposition}

\begin{proof} Let $r$ be the isotropy order  of $\v$, and  consider the diagram \eqref{diag:second return}.
let $D_1$ be the cycle on $\v$ given by $c \circ e$; note that $D_1$ has type $(2r+3, 4)$. We show that it satisfies the conditions of Lemma \ref{uniqueness test}, namely that
(i) any non-zero cycle on $\v$ of degree $4$ has length at least $2r+3$;
(ii)  any cycle on $\v$ of type $(2r+3, 4)$ is zero on $\al$ or is equal to $D_1$.
The result then follows from Proposition \ref{nilpotency test}
 (with $\psi = \v$).

	(i) Let $B$ be a  non-zero  cycle on $\v$ of degree $4$. It is the composition of two  non-zero cycles on $\v$	 of degree $2$; each of these must have length at least $r+1$. However, since $\cc^2=0$  and $B$ is non-zero, this length must be at least $(r+1) +(r+2) = 2r+3$;  in fact, if $B$ has length $2r+3$, then $B$ must be $c \circ \ep$ or
$\ep \circ c$, with $\ep$ as in Lemma \ref{epsilon}. 

	(ii) Let $B$ be a cycle on $\v$ of type $(2r+3,4)$ which is non-zero on $\al$. Then by (i)  $B$ must be $c \circ \ep$ with $\ep$ as in Lemma \ref{epsilon}, since  $\ep \circ c$ vanishes on $\al$.  In case (2) of that lemma, this $B$ is zero on $\al$ since $\cc$, and so $\epp$, vanishes on $\al$.  In cases (3) and (4), the image of $\epp$ is contained in $\image \cc$ so the composition $\cc \circ \epp$ is zero.  Hence we have case (1): $\ep = e$, so $B=c \circ e=D_1$.
\end{proof}

The proof is illustrated by diagram \eqref{diag:second return2}.
The  only possible non-zero component of the map
$\DD_1 =  \cc \circ \ee|_{\al}:\al \to \al$
 is given by the composition
\begin{equation} \label{D1}
\al \to \al_1 \to \ldots \to \al_r \to R \to \be_0 \to \be_1 \to \ldots \to \be_r \to \al\,.
\end{equation}
In \cite[p.~276]{burstall-wood}, this was shown to be holomorphic (see also Lemma \ref{e holo}), and so nilpotent when $M =S^2$;
Proposition \ref{ce} generalizes this.

Our first application is to extend a theorem of
Aithal \cite{aithal-G25} to an arbitrary Riemann surface.

\begin{proposition} \label{G2C5}
Let $\v:M \to G_2(\C^5)$ be a harmonic map of finite uniton number which is $\pa'$- and $\pa''$-irreducible.  Then $G'(\v)$ is $\pa'$-reducible.
\end{proposition}

\begin{proof} The first return map $\cc:\v\to\v$ is nilpotent as $\v$ is of finite uniton number; since $\v$ has rank $2$, we must have $\cc^2=0$.
By irreducibility, $G'(\v)$ and $G''(\v)$ again have rank $2$; since they all lie in $\C^5$, they cannot be mutually orthogonal so that $\v$ has isotropy order precisely $1$. Thus we have diagram \eqref{diag:second return2} with $r=1$ and all vertices of rank one. Now the  bundle map $\DD_1 = \cc \circ \ee|_{\al}: \al \to \al$ is nilpotent and so zero; however, all second fundamental forms  in \eqref{D1} are non-zero, except possibly $\al_1 \to R$, so this last arrow must be zero giving the conclusion.
\end{proof}

Our second application is to extend the description of harmonic maps from $S^2$ to $G_2(\C^n)$ in \cite[Theorem 3.3]{burstall-wood} to harmonic maps of finite uniton number from any Riemann surface.
This uses the idea of forward replacement (\S \ref{subsec:Grass}) to increase the isotropy order, see
\cite[Proposition 3.4]{burstall-wood} and
\cite[Theorem 7.2]{chern-wolfson}; a slightly more general and precise version is given in \cite[Proposition 3.2]{bahy-wood-G2}.

\begin{proposition} \label{G2}
Let $\v:M \to G_2(\C^n)$ be a harmonic map of finite uniton number, and finite isotropy order $r$ \  $(r \geq 1)$.
Then the harmonic map
$\wt{\v}$ obtained from $\v$ by forward replacement of the image of the first return map of $\v$ has isotropy order $r+1$.
\end{proposition}

\begin{proof} Since $\v$ has finite uniton number and rank $2$, we have $\cc^2=0$.  We use Proposition \ref{ce} to generalize the proof given in
\cite{burstall-wood, chern-wolfson} for $M=S^2$ in the case when
$\v$, $G^{(1)}(\v), \ldots, G^{(r)}(\v)$ are all $\pa'$-irreducible:

In diagram \eqref{diag:second return2}, set $\al_0=\al =$ the image of the first return map  $\cc=\cc_r(\v)$; note that $\al\subset \ker \cc$ as $\cc^2=0$. Set $\al_{r+1} = \image A'_{G^{(r)}(\v)}|_{\al_r}$; since $G^{(r)}(\v)$ is $\pa'$-irreducible, $\al_{r+1}$ and the map $\al_r \to \al_{r+1}$ are non-zero.
 Setting $\be_{r+1} = R \ominus \al_{r+1}$, the diagram \eqref{diag:second return2} becomes

\vspace{-2.5ex}
\begin{equation}
\begin{gathered}\label{diag:G2n}
\xymatrixrowsep{1.7pc}
\xymatrixcolsep{2pc}
\xymatrix{
\be_0 \ar[r]\ar[rd] & \be_1 \ar[r]\ar[rd] & \space\space\space\cdots \ar[r]\ar[rd] &\be_{r-1} \ar[r]\ar[rd] &\be_r\ar[lllld]\ar[r]\ar[rd] & \be_{r+1} \ar[r]\ar[rd] & \be_0
\\
\al_0 \ar[r]\ar[u] & \al_1 \ar[r]\ar[u] & \space\space\space			\cdots \ar[r] &\al_{r-1} \ar[r]\ar[u] & \al_r \ar[r]\ar[u] & \al_{r+1}\ar[r]\ar[u]\ar[ur] & \al_0\ar[u]
}
\end{gathered} \end{equation}
where we repeat the first column.  Note that all vertices are of rank one except possibly $\be_{r+1}$ (which could even be zero).  The harmonic map obtained from $\v$ by forward replacement of $\al$ is
$\wt\v = \be_0 \oplus \al_1$.

Then, from Proposition \ref{ce}, the  bundle map $\DD_1 = \cc \circ \ee|_{\al}$ is nilpotent, so zero.  Since all the other arrows in \eqref{D1} are non-zero, this means that
$\al_{r+1} \to \be_0$ must be zero.  Hence, $G^{r+1}(\wt\v)$ lies in $\be_{r+1} \oplus \al_0$;
this is orthogonal to $\wt\v$, showing that the isotropy order of $\wt\v$ is at least $r+1$.
Since the map $A'_{\al_0,\al_1}\circ \cc(\v)|_{\be_0}$ corresponding to the path $\be_0 \to \ldots \to \be_r \to \al_0 \to \al_1$ is non-zero, the isotropy order of $\tilde{\varphi}$ is precisely $r+1$  with $\pi_{\al_1} \circ \cc(\wt\v)|_{\be_0}=A'_{\al_0,\al_1}\circ \cc(\v)|_{\be_0}$.

If $G^{(i)}(\v)$ is $\pa'$-reducible for some  $i \in \{0,1, \ldots, r\}$, we have the same diagram with $\al_{i+1}, \ldots, \al_{r+1}$ zero.
The  map
$\ee|_{\al}$ is thus zero and the proof goes through without needing Proposition \ref{ce} --- see the proof of \cite[Proposition 3.2]{bahy-wood-G2} for more explicit diagrams  when $i \neq 0$.
\end{proof}

\begin{remark} \label{rem:irr}
 Let $\v:M \to G_k(\C^n)$ be harmonic and let $\al$ be a holomorphic subbundle of\/ $\ker(A'_{\v^{\perp}} \circ A'_{\v})$ (cf.\ \S \ref{subsec:Grass}). Set $\al_1:= A'_{\v}(\al)$.

{\rm (i)} \emph{If $\v$ is $\pa'$-irreducible or, more generally, $\rk \al_1 = \rk \al$}, then the inverse of forward replacement of $\al$ is backward replacement of $\al_1$.  In the case $k=2$, forward replacement of a line subbundle $\al$ with $A'_{\v}(\al) \neq 0$ increases (resp.\ decreases) the isotropy order by \emph{precisely} one according as $\al = \image \cc$ (resp.\ $\al \neq \image \cc$); the two operations are inverse and give bijections \cite[\S 3]{bahy-wood-G2}.

{\rm (ii)} \emph{If,  on the other hand, $\rk \al_1 < \rk \al$
so that} $\v$ is $\pa'$-reducible, the forward replacement reduces the rank of $\v$ so that it cannot be inverted by a backward replacement.
\end{remark}

We may carry out the  forward replacement operation in Proposition \ref{G2} repeatedly until we obtain a $\pa'$-reducible harmonic map $\v:M \to G_2(\C^n)$ (which must happen for dimension reasons), so that
we now need to understand such maps.  Firstly, $\rk G'(\v)$ ($= \rk A'_{\v}\,$) is zero  if and only if $\v$ is antiholomorphic.  We describe two other types of $\pa'$-reducible maps which generally have $\rk G'(\v) = 1$:

(i) A \emph{Frenet pair} is a map $\v:M \to G_2(\C^n)$ of the form
$\v = G^{(j)}(h) \oplus G^{(j+1)}(h)$ for some holomorphic map $h:M \to \CP^{n-1}$ and some $j \in \{0, 1,2,\ldots, n-2\}$ such that $G^{(j)}(h)$ and $G^{(j+1)}(h)$ are non-zero --- this is automatic if $h$ is \emph{(linearly) full}, i.e. its image lies in no proper projective subspace of $\CP^{n-1}$.  Note that a Frenet pair is strongly isotropic, i.e., of infinite isotropy order.

(ii) A \emph{mixed pair} is a map $\v:M \to G_2(\C^n)$
of the form $\v = g \oplus h$ where $h,g:M \to \CP^{n-1}$ are holomorphic and antiholomorphic, respectively, with  $h \perp g$ and $G'(h) \perp g$.
Note that $\v^{\perp}$ is strongly isotropic.

Both types of map  are harmonic \cite[\S 3]{burstall-wood} and of finite uniton number.
We have the following useful characterization of them which extends that in \cite[Proposition 3.7]{burstall-wood} to arbitrary Riemann surfaces.  Note that this holds for finite or infinite isotropy order.

\begin{proposition} \label{Frenet-mixed}
Let $\v:M \to G_2(\C^n)$ be a harmonic map of finite uniton number
with $\rk G'(\v) = 1$.  Then $\v$ is a Frenet pair or a mixed pair if and only
\begin{equation} \label{Frenet-mixed-condn}
A''_{\v}(\v \ominus \ker A'_{\v}) = 0.
\end{equation}
\end{proposition}

\begin{proof}   Let $\v:M \to G_2(\C^n)$ be a harmonic map of finite uniton number with $\rk G'(\v) = 1$, and suppose that
it satisfies \eqref{Frenet-mixed-condn}. We have the following  diagram from \cite{burstall-wood}
where $\al = \ker A'_{\v}$, $\be = \v \ominus \al$,
$R = \{\v \oplus G'(\v)\}^{\perp}$;
there is no arrow from $R$ or $G'(\v)$ to $\be$ by hypothesis.
We consider the cycle $C:\al \to \be \to G'(\v) \to \al$ of type $(3,2)$.   Clearly, for any $j \in \N$, any cycle on $\al$ of type $(3j,2j)$ is zero or equal to $C^j$, so by Proposition \ref{nilpotency test} with $\psi=\al$, $\Cc:\al \to \al$ is nilpotent. Since the vertices of $C$ are all of rank $1$, one edge must vanish giving a Frenet or mixed pair.

\vspace{-1.3ex}
\begin{equation*}
\begin{gathered}
\xymatrixrowsep{1.2pc}
\xymatrixcolsep{2pc}
\xymatrix{
\be \ar[dr] & & & \\
\al \ar[u] & G'(\v) \ar[l]\ar[r] & R \ar@/^1pc/[ll]}
\end{gathered}
\end{equation*}

\vspace{1.5ex}

A second approach is to use the basic diagram \eqref{diag:basic} on which we consider the cycle $\wt C$ on $\v$ of type $(3,2)$
given by the the self-arrow on $\v$
followed by the path $\v \to \v^{\perp} \to \v$.
With $\al$ as above, clearly, for any $j \in \N$, any cycle on $\v$ of type $(3j,2j)$  is zero on $\al$ or equal to $\wt C^j$, so by Proposition \ref{nilpotency test}  with $\psi=\v$, $\wt\Cc|_{\al} = \Cc:\al \to \al$ is again nilpotent.  (Note that $\wt \Cc|_{\al}$ is a bundle map as predicted by Lemma \ref{lem:bmap}.)

The converse is clear.
\end{proof}

Proposition 3.8 of \cite{burstall-wood} generalizes to show that any $\pa'$-reducible harmonic map $\v:M\to G_2(\C^n)$ of finite uniton number can be reduced to  one satisfying the condition \eqref{Frenet-mixed-condn} by a finite number of \emph{backward replacements}.
Then, using Propositions \ref{G2} and \ref{Frenet-mixed}, and reversing the orientation of $M$ (see \S \ref{subsec:Grass}), we obtain the following generalization of \cite[Theorem 3.3]{burstall-wood}:

\begin{theorem} \label{BuWo 3.3}
Let $\v:M \to G_2(\C^n)$ be a harmonic map of finite
uniton number and finite isotropy order from a Riemann surface.
Then there is a sequence of harmonic maps $\v_0,\ldots,\v_N:M \to G_2(\C^n)$ such that
\begin{enumerate}
\item[{\rm (i)}] $\v_0$ is holomorphic, a Frenet pair or a mixed pair$;$
\item[{\rm (ii)}] $\v_N = \v;$
\item[{\rm (iii)}] For each $i$, $0 \leq i < N$, there is a holomorphic subbundle $L_i$ of $\v_i$ such that $\v_{i+1}$ is obtained from $\v_i$ by forward replacement of $L_i$ or backward replacement of
$\v_i \ominus L_i$.
\end{enumerate}
\end{theorem}

Note that, for strongly isotropic harmonic maps, a similar result was already given in \cite[Theorem 3.9]{burstall-wood}; in that case there are no cycles to consider so that no nilpotency results are required.

The above method is adapted to studying harmonic maps from $S^2$ to $G_2(\R^n)$ (or the complex quadric $Q_{n-2}$) in \cite{bahy-wood-G2}.  These results only use the nilpotency of  $\cc \circ \ee$ and so extend as follows:  A \emph{real mixed pair} is a map $\v:M \to G_2(\R^n)$ of the form $h \oplus \ov{h}$ where $h:M \to Q_{n-2} = \bigl\{[Z] = [Z_1,\ldots,Z_n] \in \CP^{n-1}: Z_1^{\,2} + \cdots + Z_n^{\,2} = 0 \bigr\}$ is holomorphic;  such a map $\v$ is always harmonic;  note that there are no real Frenet pairs
\cite[Proposition 5.10]{bahy-wood-G2}.  If we alternate forward and backward replacement, we can keep $\v_i$ in $G_2(\R^n)$  for $i$ even leading to the following result:

\begin{theorem} \label{bahy-wood-G2}
Let $\v:M \to G_2(\R^n)$ be a harmonic map of finite
uniton number and finite isotropy order from a Riemann surface.
Then there is a sequence of harmonic maps $\v_0,\ldots,\v_N:M \to G_2(\C^n)$ such that
\begin{enumerate}
\item[{\rm (i)}] $\v_0$ is a real mixed pair$;$
\item[{\rm (ii)}] $\v_N = \v;$
\item[{\rm (iii)}] $\v_{i+1}$ is obtained from $\v_i$ by forward or backward replacement  of a line subbundle according as $i$ is even or odd.
\end{enumerate}
\end{theorem}

See \cite[Theorem 4.7]{bahy-wood-G2} for a more precise result  stated for $S^2$ but which generalizes immediately to finite uniton maps from any Riemann surface. This has been used for finding constant curvature minimal $2$-spheres in $Q_{n-2}$, for example, see \cite{JiaoLi Qn, JiaoLi Q5,LiHeJiaoQ3,PengWangXu}.  For generalizations of some of that work to harmonic maps of finite uniton number from arbitrary Riemann surfaces see \S \ref{subsec:const-curv} below.

There is an analogous result for  harmonic maps from $S^2$ into quaternionic projective space \cite[Theorem (4.7)]{bahy-wood-HPn} which equally well extends to finite uniton maps from a Riemann surface.  This likewise gives results on constant curvature and homogeneous minimal  $2$-spheres in $\HP^n$, e.g. \cite{FeiHe,HeJiao2014,HeJiao}; we will discuss some generalizations to other Riemann surfaces elsewhere.

\subsection{The cycle $c \circ e^s$}

To make further progress we generalize Proposition \ref{ce} to give a sequence of nilpotent cycles as follows.

\begin{proposition} \label{ce^s}
Let $\v:M \to G_k(\C^n)$ be a harmonic map of finite uniton number and finite isotropy order, with first return map $\cc$ of nilorder 2.
Suppose that $\al$ is any subbundle of $\v$  with $\image \cc  \subseteq \al \subseteq \ker \cc $.  For $s=1,2,\ldots$, let $\DD_s$ be the bundle map given by $\cc\circ \ee^s|_{\al}:\alpha\to \alpha$.  Suppose that for some $s \in \N$,
\begin{equation} \label{ce^t=0}
 \DD_t=0\, \text{ on } \al, \text{ for } 1 \leq t < s.
\end{equation}
Then $\DD_s:\al \to \al$ is nilpotent.
\end{proposition}

\begin{proof}
Let $r$ be the isotropy order of $\v$, and  consider the diagram \eqref{diag:second return}.
For $s=1,2,\ldots,$ let $D_s$ be the cycle on $\v$ given by $c \circ e^s$;
 then $D_s$ has type $(\ell_{r,s}, 2s+2)$ where $\ell_{r,s} := (r+1)+s(r+2)$.  We show by induction on $s$ that $D_s$ satisfies the conditions of Lemma \ref{uniqueness test}, namely that
(i) any  non-zero  cycle on $\v$ of degree $2s+2$ has length at least $\ell_{r,s}$;
(ii)  any  cycle on $\v$ of type
$(\ell_{r,s}, 2s+2)$ is zero on $\al$ or is equal to $D_s$.
The result then follows from Proposition \ref{nilpotency test}
(with $\psi=\v$).

As base of the induction, note that conditions (i) and (ii) hold for $s=1$ where $D_1 = c \circ e$,  as in the proof of Proposition \ref{ce}.

As induction hypothesis, suppose that
for some $s > 1$, (i) and (ii)  hold for all
$1 \leq t < s$. Further, suppose that \eqref{ce^t=0} holds.
We show that (i) and (ii) hold for $s$:

(i) Any non-zero  cycle $B$  on $\v$ of degree $2s+2$ is the composition of $s+1$  non-zero  cycles of degree $2$; each of these is $c$ or (*): \emph{a non-zero  cycle $\ep$ on $\v$  of degree $2$ and of length at least $r+2$}.
If $B$ contains exactly one $c$, then (i)  obviously holds because $c$ has length $r+1$ and
all other non-zero external cycles on $\v$ are longer.
Suppose now that $B$ contains two or more $c$'s.  Then $B$ is of the form
$K_2 \circ C_k \circ \cdots \circ C_1 \circ c \circ K_1$ for some $k \geq 1$, where the $K_j$ are  empty or are compositions of  cycles $\ep$ as described in (*), and each $C_j$ is of the form $c \circ \ep_t \circ \cdots \circ \ep_1$ for some $1 \leq t < s$
for some $\ep_i$ as described in (*); note that, for all $j$, $C_j \circ c$ is a subpath of $B$.

  For any particular $j$,
if all $\ep$ in $C_j$ have length exactly $r+2$, then
$C_j|_{\al}$ has length $\ell_{r,t}$.  By the induction hypothesis (ii), $C_j$  is zero on $\al$ or equal to $c \circ e^t$.  However  the latter is zero on $\al$ by \eqref{ce^t=0};
since $\image \cc \subseteq \al$,  this means that  $C_j \circ c$,  and so $B$,  is zero  on $\al$.

Hence,  one of the $\ep_i$ in $C_j$ must have length more than $r+2$ so that,  $\ell(C_j) \geq (t+1)(r+2)$.  As this holds for any $C_j$, and there is an extra $c$ in the expression for $B$, (i) holds.

(ii) Now, with notation as in (i), let $B=K_2 \circ C_k \circ \cdots \circ C_1 \circ c \circ K_1$ be a cycle on $\v$ of type $(\ell_{r,s}, 2s+2)$ which is non-zero on $\al$; note that $k$ might be zero, i.e., there may be no $C_i$'s.
  Observe that, since $\ell(B)=\ell_{r,s} =(r+1)+s(r+2)$,  $K_1$ is empty or is a composition $K_1= \ep_t \circ \cdots \circ \ep_1$ of cycles  $\ep_i$  on $\v$ of type $(r+2,2)$, with $t\leq s$. We now show that all $\ep_i$ must equal $e$ and $t=s$, implying that $B=c\circ e^s=D_s$.

Each $\ep_i$ is given by Lemma \ref{epsilon}.  By the induction hypothesis and \eqref{ce^t=0}, for all $i=1,\ldots, t$,
the composition $\cc \circ \epp_t \circ \cdots \circ \epp_{i+1}$ vanishes on $\al$, and so
on $\image \cc$.
Hence, if any $\ep_i$ is $c \circ u$ or $\widehat c$, then $K_1$, and so $B$, are zero on $\al$.
So each $\ep_i$ is $u \circ c$ or $e$.  Suppose that
$\ep_1 = u \circ c$; then, since $\al \subseteq \ker \cc$, $B$ is zero on $\al$, so $\ep_1 = e$.  Suppose next that
$\ep_2 = u \circ c$; then $B$ starts with $c \circ e$ which is zero on $\al$ by hypothesis \eqref{ce^t=0}.  Hence $\ep_2 = e$.

Similarly all $\ep_i$ must equal $e$, so that $K_1= e^t$. But since, by the induction hypothesis, $\cc\circ \ee^t=0$ on $\al$ for $1\leq t<s$, we must have $t=s$ and $K_1=e^s$; so $B =c\circ e^s$ as required.
\end{proof}

We can apply the above nilpotent cycle to extend Proposition \ref{G2} to other Grassmannians as follows.  We need some preliminaries; the next two lemmas do not require finite uniton number.

\begin{lemma} \label{e holo}
Let $\v:M \to G_k(\C^n)$ be harmonic of finite isotropy order, with first return map $\cc$ of nilorder $2$.
Let $\al$ be a holomorphic subbundle of $\v$ with
\begin{equation} \label{al condn}
\image \cc \subseteq \al \subseteq \ker \cc
\end{equation}
and set $\be = \v \ominus \al$.  Then $\pi_{\be} \circ \ee: \al \to \be$ is holomorphic.
\end{lemma}

\begin{proof}
With notation as in diagram \eqref{diag:second return2},
$\pi_{\be} \circ \ee: \al \to \be$ is the composition of second fundamental forms $\al \to \al_1 \to \ldots \to \al_r \to R \to \be$.  These are all holomorphic by \cite[Proposition 1.5]{burstall-wood}.
\end{proof}

Given a holomorphic subbundle $\al$ of $\v$ satisfying \eqref{al condn}, for $t=0,1,\ldots$ \
set $E^t_{\al} = \spa\{\ee^i(\al): i =0,1,\ldots,t\}$
so that $E^0_{\al} = \al$.

\begin{lemma} \label{E}
 Assume the same hypotheses as in Lemma \ref{e holo}.
\begin{enumerate}
\item[{\rm (i)}] Suppose that $E^{t-1}_{\al}$ is a holomorphic subbundle of\/ $\ker \cc$  for some $t \geq 1$.  Then $E^t_{\al}$ is a holomorphic subbundle of $\v$.
\item[{\rm (ii)}] Suppose that $E^t_{\al}$ is a holomorphic subbundle of $\ker \cc$ for all $t \in \{0,1,\ldots\}$.  Then
$E^{\infty}_{\al} := \bigcup_{t=0}^{\infty}E^t_{\al} = \spa\{\ee^i(\al): i =0,1,\ldots\}$ is a holomorphic subbundle of\/ $\ker \cc$ and so of $\v$.
\end{enumerate}
\end{lemma}

\begin{proof}
(i)  Since $\image \cc \subseteq E^{t-1}_{\al} \subseteq \ker \cc$,
 we get a diagram like \eqref{diag:second return2},
 with $\al$ replaced by $E^{t-1}_{\al}$ and $\be$ by $\v \ominus E^{t-1}_{\al}$, to which we can apply Lemma \ref{e holo}.   By that lemma,
$\widehat{\ee} := \pi_{\v \ominus E^{t-1}_{\al}} \circ \ee:E^{t-1}_{\al} \to \v \ominus E^{t-1}_{\al}$ is holomorphic, so after filling out zeros, its image, $\image{\widehat{\ee}}$, is a holomorphic subbundle of $\v \ominus E^{t-1}_{\al}$.  It easily follows that $E^t_{\al} =  E^{t-1}_{\al} \oplus \image{\widehat{\ee}}$ is a holomorphic subbundle of $\v$.

(ii) We have $E^{\infty}_{\al} = E^t_{\al}$ for some $t$, so the result follows from part (i).
\end{proof}

To apply this in the sequel, note that $\cc:\v \ominus \ker \cc \to \image \cc$ is holomorphic and it follows from linear algebra that it is an isomorphism on each fibre except for those fibres at the isolated points where $\rk\cc$ drops.

\begin{proposition} \label{Gk nil 2 rk 1}
Let $\v:M \to G_k(\C^n)$ be a harmonic  map of finite uniton number and finite isotropy order, with first return map $\cc$ of nilorder $2$ and rank $1$.
Then there exists a holomorphic subbundle $\al$ of $\v$ with
$\image \cc \subseteq \al \subseteq \ker \cc$ such that forward replacing $\al$ increases the isotropy order of $\v$.
\end{proposition}

\begin{proof}
 Denote the isotropy order of $\v$ by $r$.
Suppose there is an $s \in \N$ such that $\ee^s(\image \cc)$ does not lie in $\ker \cc$; choose the least such $s$.  Then, by Proposition \ref{ce^s}, $\cc \circ \ee^s = \cc \circ \pi_{\v \ominus \ker \cc} \circ \ee^s:\image \cc \to \image \cc$ is nilpotent and so zero. Now, as explained above, $\cc:\v \ominus \ker \cc \to \image \cc$ is an isomorphism (away from isolated points); this means that $\pi_{\v \ominus \ker \cc} \circ \ee^s(\image \cc)$ must be zero, so  $\ee^s(\image \cc)$ lies in $\ker \cc$, a contradiction.
Hence there is no such $s$ and,  by Lemma \ref{E}, $\al := E^{\infty}_{\image\cc}$ is a holomorphic subbundle of  $\v$,
with $\image\cc \subseteq \al \subseteq \ker \cc$, which is closed under $\ee$.

Let $\wt\v$ be obtained from $\v$ by forward replacement of $\al   = E^{\infty}_{\image \cc}$.
 Then since this $\al$ is closed under $\ee$, the arrow  $\al_{r+1}\to\be_0$ in diagram \eqref{diag:G2n} is zero. So $G^{(r+1)}(\wt\v) = \image(A'_{G^{(r)}(\wt\v)})$ lies in
 $\be_{r+1} \oplus \al_0$  which is orthogonal to $\wt\v$, showing that $\wt\v$ has isotropy order at least $r+1$.
\end{proof}

For $\cc$ of rank more than one, we will find forward replacements which reduce the rank of $\cc$.  These use the following lemma which does not require finite uniton number.

\begin{lemma} \label{reduce rank}
Let $\v:M \to G_k(\C^n)$ be harmonic {of finite isotropy order $r$ ($r \geq 1$), with first return map $\cc$} of nilorder $2$.  Let $\al$ be a holomorphic subbundle of $\v$ with
$\image \cc \subseteq \al \subseteq \ker \cc$ and let $\wt\v$ be obtained from $\v$ by forward replacing $\al$, thus $\wt\v = \be \oplus \al_1$ where $\be = \v \ominus \al$ and $\al_1 = A'_{\v}(\al)$.    With $\al_i$ and $R$ as in Example \ref{ex:c2=0}, define $\wt \cc_r:\wt\v \to \wt\v$ by $\wt \cc_r|_{\be} = 0$ and
$\wt \cc_r|_{\al_1}= $ the composition $\al_1 \to \cdots \to \al_r \to R \to \be$; note that   $\image(\wt\cc_r)= \wt \cc_r(\al_1) = \pi_{\be} \circ \ee(\al)$. If\/ $\wt \cc_r$ is zero then $\wt\v$ has isotropy order greater than  $r$; otherwise $\wt\v$ has isotropy order $r$ and first return map $\wt \cc_r$ of nilorder $2$.
\end{lemma}

\begin{proof} The result is clear from diagram \eqref{diag:second return2}.
\end{proof}

Note that the lemma still holds when $\rk \al_1 < \rk \al$ (in which case $\rk \wt\v < \rk \v$); in the extreme case $\al_1=0$, we have $\wt\v = \be$ and $\ee(\al)=0$.
Combining the lemma with the nilpotency of $\cc\circ\ee$ enables us to simplify maps into $G_4(\C^n)$ as follows.

\begin{proposition} \label{G4}
Let $\v:M \to G_4(\C^n)$ be a harmonic map of finite uniton number and finite isotropy order $r$ ($r \geq 1$), with first return map $c$ of nilorder $2$ and rank $2$.
Then the harmonic map
$\wt{\v}$ obtained from $\v$ by forward replacement of the image of the first return map of $\v$ \emph{either} has isotropy order $r$ and first return map of nilorder $2$ and rank $1$, \emph{or} has isotropy order more than $r$.
\end{proposition}

\begin{proof}  With $\al_0=\al =$ the image of the first return map, we get diagram \eqref{diag:second return2}.  The harmonic map obtained from $\v$ by forward replacement of $\al$ is
$\wt\v = \be_0 \oplus \al_1$ where $\al_1 = \image A'_{\v}|_{\al}$.

Then, from Proposition \ref{ce}, $\cc \circ \ee = \cc \circ \pi_{\beta} \circ \ee:\al \to \al$ is nilpotent, so of rank at most $1$.  Since $\cc: \beta = \v \ominus \ker \cc \to \al = \image \cc$ is an isomorphism (away from isolated points), this means that
$\pi_{\be}\circ \ee(\al)$ must have rank at most $1$. The result follows from Lemma \ref{reduce rank}.
\end{proof}

We remark that the same result applies to $G_{2j}(\C^n)$ with a first return map of nilorder 2 and rank $j$ where $j \in \N$.

To deal with $G_5(\C^n)$, we introduce a new nilpotent cycle:

\begin{proposition} \label{cece^2}
Let $\v:M \to G_k(\C^n)$ be a harmonic map of finite uniton number and finite isotropy order,  with first return map $\cc$ of nilorder 2 and rank 2.
Set $\eta = \image(\cc \circ \ee)|_{\image \cc}$  and define a bundle map $\FF: \eta \to \eta$ by
$$
 \FF = \cc \circ \ee \circ \cc \circ \ee^2|_{\eta}.
$$
Then $\FF$ is nilpotent, and so is zero.
\end{proposition}

\begin{proof}
Let $r$ be the isotropy order of $\v$, and  consider the diagram  \eqref{diag:second return}. Let $F$ be the cycle on  $\v$ given by $c \circ e \circ c \circ e^2$. Note that  $\FF$ maps $\eta$ to $\eta$.
 If $F$ is zero, there is nothing to prove; otherwise, $\eta$ is non-zero. By Proposition \ref{ce}, $\cc \circ \ee:\image \cc \to \image \cc$ is nilpotent; since $\image \cc$ has rank $2$, this means that
\begin{equation} \label{(ce)^2}
(\cc \circ \ee)^2|_{\image \cc} = 0 \text{ and }
\eta = \image (\cc \circ \ee)|_{\image \cc} = \ker (\cc \circ \ee)|_{\image \cc} \text{ has rank } 1. 
\end{equation}

Now $F$ is of type $(\ell(F), 10)$ where $\ell(F)= 2(r+1)+3(r+2)$. We show that $F$ satisfies the conditions (i) and (ii) of Lemma \ref{uniqueness test}
with $\al$ taken to be $\eta$:

(i) Let $B$ be a  non-zero cycle on $\v$ of degree 10; we must show that $\ell(B) \geq \ell(F)$. Now $B$ is the composition of five  non-zero  cycles on $\v$ of degree 2, each of which is $c$ or a  (*) non-zero cycle $\ep_i$  on $\v$ of length at least $r+2$.  Since $\cc^2=0$, the $c$'s must be separated, so there can be at most three.

If $B$ has precisely three $c$'s we must have
$B=c \circ \ep_2 \circ c \circ \ep_1 \circ c$  for some $w_i$ as in (*).  If this has length less than $\ell(F)$, then the $\ep_i$ must have length $r+2$ and so are as described in Lemma \ref{epsilon}.   Possibilities (2), (3) and (4) give $B=0$, so we must have $\ep_1 = \ep_2 = e$ giving $B=c \circ e \circ c \circ e \circ c$.  But then by \eqref{(ce)^2}, $B=0$.
Hence, $B$ can contain at most two $c$'s, so that $\ell(B) \geq \ell(F)$.

(ii) Now let $B$ be a  cycle on $\v$ of type $(\ell(F), 10)$ which is non-zero on $\eta$.  We must show that $B=F$.  As in (i), it is the composition of two $c$'s and three cycles $\ep_i$  on $\v$ of length $r+2$.  Furthermore, it cannot start with $c \circ \ep_1$ as this is zero on $\eta$\, --- this follows from $\cc^2=0$ if $\ep_1 = u \circ c$, $c \circ u$ or $\widehat c$, and follows from $\eta= \ker(\cc \circ \ee)$ if
$\ep_1 = e$.
Hence $B = c \circ \ep_3 \circ c \circ \ep_2 \circ \ep_1$. Now $\ep_3 = e$ as, otherwise, $c \circ \ep_3 \circ c = 0$. Next, note that $\ep_1$ cannot be $u \circ c$ as $\cc|_{\eta} = 0$.
Also, $\ep_2 = e$ or $u \circ c$ as the other possibilities give zero.

Suppose that $\ep_2 = u \circ c$. Then, if $\ep_1 = c \circ u$ or $\widehat c$, $B=0$. So $\ep_1=e$ but then, by $\cc \circ \ee|_{\eta} = 0$, we have $\BB|_{\eta}=0$.  Hence $\ep_2 = e$ and $\ep_1 = c \circ u$, $\widehat c$ or $e$.  In the first two cases, $\image \BB \subseteq (\cc \circ \ee)^2(\image \cc) = 0$.
Hence $\ep_1=e$ so that $B=F$.
\end{proof}

We apply this new nilpotent cycle to understand maps into $G_5(\C^n)$:

\begin{proposition} \label{G5} Let $\v:M \to G_5(\C^n)$ be a harmonic map of finite uniton number and finite isotropy order $r$ \ ($r \geq 1$), with first return map $\cc$ of nilorder 2 and rank 2.  Then there is a holomorphic subbundle $\al$ of $\v$ with $\image \cc \subseteq \al \subseteq \ker \cc$ such that the harmonic map
$\wt{\v}$ obtained from $\v$ by forward replacement of $\al$ \emph{either} has isotropy order $r$ and first return map of nilorder $2$ and rank $1$, \emph{or} has isotropy order more than $r$.
\end{proposition}

\begin{proof}
It follows from Lemma \ref{E}(i) that we have a holomorphic subbundle  $E^1_{\image \cc} = \spa\{\image \cc,\, \ee(\image \cc)\}$ of $\v$.

(a) Suppose that $E^1_{\image \cc}$ has rank at most $3$.
Set $\al = \image \cc$ and $\be=\v \ominus \al$.  Then, $\pi_{\be} \circ \ee(\al)$  has rank at most $1$ and we are done by Lemma \ref{reduce rank}.

(b) Otherwise, $E^1_{\image \cc}$ has rank $4$.  Since $\ker \cc$ has rank $3$ we must have $\pi_{\v \ominus \ker \cc} \circ \ee(\image \cc)$ non-zero.  Since $\cc:\v \ominus \ker \cc \to \image \cc$ is an isomorphism (away from isolated points), this implies that
$(\cc \circ \ee)(\image \cc) \neq 0$.  Writing  $\eta= (\cc \circ \ee)(\image \cc) \subseteq \image \cc$ we have from \eqref{(ce)^2}, $(\cc \circ \ee)(\eta) = 0$ i.e., $\ee(\eta) \subseteq \ker \cc$.

If we had $\ee(\eta) \subseteq \image \cc$, then $E^1_{\image \cc}$ would have rank at most $3$.  Hence,
$\spa\{\image \cc, \ee(\eta)\} = \ker \cc$.  Set $\psi = \image \cc \ominus \eta$ so that $\ker \cc = \spa\{\psi, \eta, \ee(\eta)\}$. Note that
$(\cc \circ \ee)(\psi) = (\cc \circ \ee)(\image \cc) = \eta$.

Set $\al = \ker \cc$ and $\be = \v \ominus \al$. By Proposition \ref{cece^2},
$\cc \circ \ee^2(\eta) \subseteq \ker(\cc \circ \ee)|_{\image \cc}
= \eta = (\cc \circ \ee)(\psi)$; 
since $\cc:\be=\v \ominus \ker \cc \to \image \cc$ is an isomorphism (away from isolated points), this implies that 
$\pi_{\be}\circ \ee(\psi)$ has rank one, and
$ \pi_{\be} \circ \ee^2(\eta) \subseteq
\pi_{\be}\circ \ee(\psi)$.
Also, since $\ee(\eta) \subseteq \ker \cc$ we have $\pi_{\be} \circ \ee(\eta)  = 0$.

It follows that 
$\pi_{\be} \circ \ee(\al)$ has image in the rank one subbundle $\pi_{\be} \circ \ee(\psi)$.
So, again, $\pi_{\be} \circ \ee(\al)$  has rank at most $1$  and we are done by Lemma \ref{reduce rank}.
\end{proof}

\subsection{First return map of nilorder greater than 2}

For first return maps of nilorder greater than 2 we need a new sequence of nilpotent cycles:

\begin{proposition} \label{nilorder p}
Let $\v:M \to G_k(\C^n)$ be a harmonic map of finite uniton number and finite isotropy order,  with first return map $\cc$ of nilorder $p \geq 2$.  Let $\al$ be a holomorphic subbundle of $\v$ with  $\image \cc^{p-1}\subseteq \alpha  \subseteq \ker \cc$.  Then the bundle map defined by
 $\EE_{p-1} = \cc^{p-1} \circ \ee|_{\al}:\al \to \al$ is nilpotent.
\end{proposition}

\begin{proof}
Let $r$ be the isotropy order of $\v$, and  consider the diagram \eqref{diag:second return}.
Let $E_{p-1}$ be the cycle on $\v$ given by $c^{p-1} \circ e$\,;
this cycle is of type $(\ell(E_{p-1}), 2p)$ where $\ell(E_{p-1})= (p-1)(r+1)+(r+2)$.
 If $E_{p-1}$ is zero, there is nothing to prove; so assume that it is non-zero.  We show that $E_{p-1}$ satisfies the
conditions of Lemma \ref{uniqueness test}, namely that
(i) any  non-zero  cycle on $\v$ of degree $2p$ has length at least $\ell(E_{p-1})$;
(ii)  any  cycle on $\v$ of type $(\ell(E_{p-1}), 2p)$  which is non-zero on $\al$ is equal to $E_{p-1}$.
Then the result follows by Proposition \ref{nilpotency test}.

For (i), note that $B$ is the composition of $p$  non-zero cycles on $\v$ of degree $2$.  Since $\cc^p=0$, at most $p-1$ of these can be $c$, and the others have length at least $r+2$, making $\ell(B) \geq \ell(E_{p-1})$ as desired.

For (ii), such a  cycle $B$ is the composition of $p-1$ $c$'s and one cycle $\ep$ as described in Lemma \ref{epsilon}. Since $\cc|_{\al} =0$, $B$ must be $c^{p-1} \circ \ep$.
If $\ep = u \circ c$ then $B=0$ by $\cc|_{\al}=0$.  If $\ep = c \circ u$ or $\widehat c$ then, recalling from Lemma \ref{epsilon} that $\image\widehat\cc \subseteq \image \cc$,  $B=0$ by $\cc^p=0$.	
Hence $\ep=e$ and $B=E_{p-1}$.
\end{proof}

The next two lemmas do not require finite uniton number.

\begin{lemma} \label{ker c^i}
Let $\v:M \to G_k(\C^n)$ be harmonic of finite isotropy order $r$ with first return map $\cc$ of  finite nilorder  $p \geq 2$.
Then the following holds$:$
\begin{enumerate}
\item[{\rm (i)}]  With $\cc^0$ equal to the identity map, we have a nested sequence of holomorphic subbundles of $\v$\,$:$
\begin{equation*} %\label{nested}
\v = \ker \cc^p \supset \cdots \supset \ker \cc^{i+1} \supset \ker \cc^i \supset \cdots \supset \ker \cc \supset  \ker \cc^0 = {\mathbf 0}.
\end{equation*}
\item[{\rm (ii)}] For $i=0,1,\ldots, p-1$,
$\cc$ maps $\ker \cc^{i+1}$ into $\ker \cc^i$ and
$\cc^{-1}(\ker \cc^i) = \ker \cc^{i+1}$.
\item[{\rm (iii)}] Set $\ga^i = \ker \cc^{i+1} \ominus \ker \cc^i$ \ $(i=0,1,\ldots, p-1)$ so that
$\v = \bigoplus_{i=0}^{p-1} \ga^i$.
Then, for $1 \leq i \leq p-1$, $\pi_{\ga^{i-1}} \circ \cc$ maps $\ga^i$ into $\ga^{i-1}$ injectively.  Hence the dimensions of the $\ga^i$ form a decreasing sequence of positive numbers.
\item[{\rm (iv)}] The map $\cc^{p-1}$ restricts to an isomorphism (away from isolated points) from $\ga^{p-1} = \v \ominus \ker \cc^{p-1}$ to $\al:= \image \cc^{p-1}$.  Further $\al \subseteq \ga^0 = \ker \cc$.
\item[{\rm (v)}] The kernel of the composition $\v \to G'(\v) \to \cdots \to G^{(r)}(\v)$ lies in $\ga^0 = \ker \cc$.
\end{enumerate}
\end{lemma}

\begin{proof}
The holomorphicity follows from the holomorphicity of $\cc$; the rest is linear algebra.
\end{proof}

 We now give a version of Lemma \ref{reduce rank} for $p>2$. Note that the formula for $\wt \cc_r$ is more complicated.
 We define the $\al_i$ and $R$ as in Example \ref{ex:c2=0};
 note that we must add an arrow from $\be_r$ to $\be$ to  diagram \eqref{diag:second return2}; however we will not use this diagram.

\begin{lemma} \label{reducenilorder}
Let $\v:M \to G_k(\C^n)$ be a harmonic map of isotropy order $r$ \ ($r \geq 1$) with first return map $c$ of nilorder  $p>2$.  Let $\al$ be a holomorphic subbundle of $\v$ with
$\image \cc^{p-1} \subseteq \al \subseteq \ker \cc$.

Let $\wt\v$ be obtained from $\v$ by forward replacing $\al$, thus $\wt\v = \be \oplus \al_1$ where $\be = \v \ominus \al$.
Define $\wt \cc_r:\wt\v \to \wt\v$ by $\wt \cc_r|_{\be}
 = \pi_\beta\circ \cc|_{\be}$ and $\wt \cc_r|_{\al_1}= $
 the composition $\al_1 \to \cdots \to \al_r \to R \to \be$;
 note that $\wt \cc_r(\al_1) = \pi_{\be} \circ \ee(\al)$.
Then $\wt\v$ has isotropy order exactly $r$, and $\wt\cc_r$ is the first return map
 of $\wt\v$. Moreover, $\wt \cc_r$ has nilorder at most $p$
 and $\image\wt \cc_r^{p-1} = \wt \cc_r^{p-1}(\al_1)=\pi_{\be}\circ \cc^{p-2} \circ \ee(\al)$.
\end{lemma}

\begin{proof} Consider the nested sequence of holomorphic subbundles of $\v$:
 \begin{equation}\label{nestednilg}
 \v = \ker \cc^p \supset \cdots \supset \ker \cc^{i+1} \supset \ker \cc^i \supset \cdots \supset \ker \cc \supset \alpha \supset 0.
 \end{equation}
Set $\delta=\ker \cc \ominus \alpha$ and, as in Lemma \ref{ker c^i}, define the subbundles $\ga^i = \ker \cc^{i+1} \ominus \ker \cc^i$ \ $(i=0,\ldots, p-1)$ so that
$\ga^0=\alpha\oplus \delta$ and  $\v = \bigoplus_{i=0}^{p-1} \ga^i$.

Set $\gamma_0^i = \gamma^i$ \ ($i=0, 1,\ldots, p-1$),
 and define subbundles $\ga^i_{j}$ of $G^{(j)}(\v)$ for $j=1,\ldots, r$ inductively by
$\bigoplus_{k \leq i}\ga^k_{j}
= \image A'_{G^{(j-1)}(\v)}|_{\bigoplus_{k \leq i} \ga^k_{j-1}}$;  note that each $\bigoplus_{k \leq i}\ga^k_{j}$ is a holomorphic subbundle of $G^{(j)}(\v)$ and
$G^{(j)}(\v)=\bigoplus_{i=0}^{p-1}\ga^i_{j}$.  Finally, set $\delta_j = \ga^0_j \ominus \al_j$ so that $\delta_0 = \delta$.

From Lemma \ref{ker c^i}(v), some  $\alpha_j$ and $\delta_j$ may be zero; however, for $i\geq 1$,   all the $\gamma_j^i$ are non-zero,  in fact the compositions $ \ga_0^i \to  \cdots \to \ga^i_k$
with $0 < k \leq r$ are isomorphisms (away from isolated points).

 We have now the following
 refinement of diagram \eqref{diag:second return}:
 \begin{equation}\label{diag:second returnp}
\begin{gathered}
\xymatrixrowsep{0.6pc}
\xymatrixcolsep{1.7pc}
\xymatrix{\v & G^{(1)}(\v) & \ldots & G^{(r)}(\v) & R
	 &\v \\
\gamma^{p-1}  \ar[r] & \ar[r] & \ar[r] & \gamma^{p-1}_r  \ar@{-}[r]\ar[llld]\ar[llldd]\ar[lllddd]\ar[llldddd] & \ar@{-}[r]
	 & \gamma^{p-1} &
\\
\vdots  \ar[r]\ar@{-}[u] & \ar[r]\vdots \ar[r]\ar@{-}[u] &\ar[r]\vdots \ar[r]\ar@{-}[u] &\vdots \ar@{-}[r]\ar@{-}[u]\ar[llld]\ar[llldd]\ar[lllddd] &  \vdots \ar@{-}[r]\ar@{-}[u]
	 &\vdots\ar@{-}[u] &
\\
\gamma^1  \ar[r]\ar@{-}[u] &  \ar[r]\ar@{-}[u] & \ar[r]\ar@{-}[u] & \gamma^1_r  \ar@{-}[r]\ar@{-}[u]\ar[llldd]\ar[llld] &   \ar@{-}[r]\ar@{-}[u]
	 &\gamma^1 \ar@{-}[u] &
\\
\delta \ar@{-}[r]\ar@{-}[u] &  \ar@{-}[r]\ar@{-}[u] & \ar@{-}[r]\ar@{-}[u] & \delta_r  \ar@{-}[r]\ar@{-}[u] & \ar@{-}[r]\ar@{-}[u]
	 &\delta\ar@{-}[u]&
\\
\alpha  \ar[r]\ar@{-}[u] & \alpha_1 \ar[r] \ar@{-}[u] & \ar[r]\ar@{-}[u] & \alpha_r  \ar[r]\ar@{-}[u] &\alpha_{r+1} \ar@{-}[r]\ar@{-}[u]\ar[ruuuu]\ar[ruuu]\ar[ruu]\ar[ru]
	 &\alpha\ar@{-}[u]&
}
\vspace{2.5ex}
\end{gathered} \end{equation}
where we repeat the first column and set $\alpha_{r+1}=\image A'_{G^{(r)}(\v)}|_{\alpha_{r}}$.  Note that only the arrows which have a possible non-zero contribution to $\wt \cc_r$ and $\wt \cc_r^{p-1}$ are shown.
 The result is now clear from this diagram, noting that $\beta=  \v \ominus \al =
\delta\oplus \bigoplus_{i=1}^{p-1}\ga^i$. 
More precisely, $\wt \cc_r$ given in the statement of the theorem is non-zero on $\beta$,  since, if $\pi_{\be} \circ c|_{\be}$ were zero, $c^2$ would be zero; alternatively note that, by  Lemma \ref{ker c^i}(iii), the arrows $\gamma_r^i\to \gamma^{i-1}$ give injections.
Thus $\wt\v$ has isotropy order $r$ and $\wt \cc_r$ is clearly
its first return map.

Next, since
$\wt \cc_r|_{\be} = \pi_{\be} \circ \cc|_{\be}$, we have
$\wt \cc_r^{p-1}|_\be =  \pi_{\be} \circ \cc^{p-1}|_{\be}$,
which is zero from $\image \cc^{p-1} \subseteq \al = \v \ominus \be$.
On the other hand, 
$\wt \cc_r^{p-1}(\al_1)
= \wt \cc_r^{p-2} \circ \pi_{\be} \circ \ee(\al)
= \pi_{\be} \circ \cc^{p-2} \circ \pi_{\be} \circ \ee(\al)
= \pi_{\be} \circ \cc^{p-2} \circ  \ee(\al)$, since $\al \subseteq \ker \cc \subseteq \ker \cc^{p-2}(\al)$ when $p >2$.
\end{proof}

 We now use these lemmas to simplify some more harmonic maps.

\begin{proposition} \label{Gk nil k,k-1}
Let $\v:M \to G_k(\C^n)$ be a harmonic map of finite uniton number and finite isotropy order, with first return map $\cc$ of nilorder $p >2$ where $p=k$ or $k-1$.
Then there exists a holomorphic subbundle $\al$ of $\v$
such that forward replacing $\al$ decreases the nilorder of $c$.
\end{proposition}

\begin{proof}
 We use the notation of Lemmas \ref{ker c^i} and \ref{reducenilorder}; in particular,  $\al = \image \cc^{p-1}$ and $\ga^{p-1} = \v \ominus \ker \cc^{p-1}$. By Lemma \ref{ker c^i}, $\ga^{p-1}$, and so $\al$, is of rank one, as are all $\ga^i$, except that $\ga^0  =\ker \cc$ is of rank $2$ when $p=k-1$.

By Proposition \ref{nilorder p}, $\cc^{p-1} \circ \ee:\al \to \al$ is nilpotent and so zero; hence  $\ee(\al) \subseteq \ker \cc^{p-1}$.
We claim that
\begin{equation} \label{sum}
\ker \cc^{p-1} = \ker \cc^{p-2} + \image \cc\,.
\end{equation}
To see this note that the right-hand side is certainly a subset of the left-hand side.
{}From $\rk(\image \cc) + \rk(\ker \cc) = k$ we get that
$\rk(\image \cc)$ equals $k-1$ if $p=k$, or $k-2$ if $p=k-1$.

\emph{Case (a).} If $p=k$, $\image \cc$ and $\ker \cc^{p-1}$ both have rank $k-1$, so they must be equal and \eqref{sum} trivially holds.

\emph{Case (b).} If $p=k-1$, then $\image \cc$ and $\ker \cc^{p-2}$ both have rank $k-2$, but $\ker \cc^{p-1}$ has rank $k-1$.
Suppose that \eqref{sum} does not hold.  Then $\image \cc$ and $\ker \cc^{p-2}$ and their sum must all have rank $k-2$ and so are equal.  But this  implies that $\image \cc \subseteq \ker \cc^{p-2}$, i.e.\  $\cc^{p-1}=0$, in contradiction to the nilorder being $p$.
Hence \eqref{sum} holds.

In both cases, it follows that  $\ee(\al) \subseteq \ker \cc^{p-2} + \image \cc$,
hence $\cc^{p-2} \circ \ee(\al) \subseteq \cc^{p-2}(\image \cc) \subseteq \image \cc^{p-1} = \al$.
 Denote the isotropy order of $\v$ by $r$ and
define $\wt \cc_r$  as in Lemma \ref{reducenilorder}.
  Then $\wt \cc_r$ is non-zero, $\wt\v$ has isotropy order exactly $r$, and $\wt \cc_r$ is the first return map of $\wt\v$. By Lemma \ref{reducenilorder},
$\wt \cc_r^{p-1}$ has image $\wt \cc_r^{p-1}(\al_1)
= \pi_{\be} \circ \cc^{p-2} \circ \ee(\al) =0$, so that
$\wt \cc_r^{p-1} = 0$;
thus, forward replacing $\al$ decreases the nilorder of $\cc$.
\end{proof}

As before we can combine cycles of the above type to obtain a new nilpotent cycle, cf.\ Proposition \ref{cece^2}.

\begin{proposition} \label{c^2ece}
Let $\v:M \to G_5(\C^k)$ be a harmonic map of finite uniton number and finite isotropy order, with first return map $c$ of nilorder 3. Set $\al = \image \cc^2$.   Then the bundle map defined by
$$
\GG= \cc^2 \circ \ee \circ \cc \circ \ee|_{\al}:\al \to \al
$$
is nilpotent, and so is zero.
\end{proposition}

\begin{proof}
 Let $r$ be the isotropy order of $\v$, and  consider the diagram \eqref{diag:second return}.
Let $G$ be the cycle on  $\v$ given by
$c^2 \circ e \circ c \circ e$.
   If $G$ is zero, there is nothing to prove. Otherwise,
 defining the $\ga^i$ as in Lemma \ref{ker c^i}, by part (iii) of that lemma, the ranks of the $\ga^i$ are decreasing for $i=0,1,2$ and add up to $5$. Hence $\ga^2 = \v \ominus \ker \cc^2$ has rank $1$; by Lemma \ref{ker c^i}(iv), $\al$ also has rank $1$.
 The cycle $G$ is of type $(\ell(G),10)$ where $\ell(G)=3(r+1)+2(r+2)$; we show that $G$ satisfies the conditions (i) and (ii) of Lemma \ref{uniqueness test}.

(i) Let $B$ be a  non-zero  cycle on $\v$ of degree 10; we must show that $\ell(B)\geq \ell(G)$. The cycle $B$ is the composition of five non-zero cycles on $\v$ of degree 2, each of which is $c$ or a cycle $\ep$ of length at least $r+2$.  Now, since $\cc^3=0$,  $B$  contains at most four $c$'s.   If it contains three or fewer $c$'s, then its length is at least $\ell(G$). If, instead, it contains four $c$'s, then $B=c^2\circ \ep\circ c^2$; further, if this $B$ has length less than $\ell(G)$, then $\ep$ has length $r+2$ and so is one of the cycles describe in Lemma \ref{epsilon}. By $\cc^3=0$, possibilities (2), (3) and (4) give
$ B=0$,  hence we must have $\ep=e$. But $\al$ has rank 1 and, by Proposition  \ref{nilorder p}, the  $\cc^2\circ \ee|_{\al}$ is nilpotent and so zero; thus, when $\ep=e$, we also have $B=0$.
Hence, $\ell(B)\geq \ell(G)$.

(ii) Let $B$ be a  cycle on $\v$ of type $( \ell(G),10)$  which is non-zero on $\al$. We show that $B=G$. As before, $B$ is the composition of three $c$'s and two cycles $\ep_i$ of length $r+2$. Since $\cc^3=0$ we have
\begin{equation} \label{al-ker}
\al = \image \cc^2 \subseteq \ker \cc ,
\end{equation}
and there are two possibilities: either (a) $B=c\circ \ep_2\circ c^2 \circ \ep_1$ or (b) $B=c^2\circ \ep_2\circ c \circ \ep_1$, where $\ep_1,\ep_2$ are as described in Lemma \ref{epsilon}.

In case (a), if $\ep_1$ equals $u\circ c$, $c\circ u$ or $\widehat c$, then $\BB|_{\al}=0$ by \eqref{al-ker}.  Moreover, if $\ep_1=e$ then $\BB|_{\al}=0$ because, as above, $\cc^2\circ \ee$ is zero on $\alpha$.

In case (b), if $\ep_2$ equals $c \circ u$ or $\widehat c$, then $\BB|_{\al}=0$ because $\cc^3=0$. If $\ep_2$ equals $u\circ c$, then  $B=c^2\circ u\circ c^2 \circ \ep_1$. However, as in case (a), all four possibilities for $\ep_1$ give $\BB|_{\al}=0$.
Hence we must have $\ep_2=e$, so $B=c^2\circ e\circ c \circ \ep_1$. Since $\image \cc\circ \widehat \cc  \subseteq \alpha$ and $\cc^2\circ \ee$ is zero on $\alpha$,
the cases $\ep_1=\widehat c$ or $c \circ u$ give $B|_{\al}=0$.
If $\ep_1 = u \circ c$, $\BB|_{\al}=0$ by \eqref{al-ker}.
Hence $\ep_1=e$, so $B=c^2\circ e\circ c \circ e=G$.
\end{proof}

We give an application of this.

\begin{proposition} \label{G5 nil 3}
Let $\v:M \to G_5(\C^n)$ be a harmonic map of finite uniton number and finite isotropy order, with
first return map $\cc$ of nilorder $3$.  Then $\cc$ has rank $2$ or $3$.
\begin{enumerate}
  \item[(a)] If $\cc$ has rank 2, then forward replacing $\image \cc^2 $  gives a harmonic map of the same isotropy order as $\v$ but with first return map of nilorder $2$.
\item[(b)] If $\cc$ has rank 3,  then forward replacing $\ker \cc$ gives a harmonic map of the same isotropy order as $\v$ but first return map of nilorder $3$ and rank less than $3$.
\end{enumerate}
\end{proposition}

\begin{proof} As usual, denote the isotropy order of $\v$ by $r$.  Consider the nested sequence of holomorphic subbundles of $\v$:
 \begin{equation*} %\label{nestednil3}
 \v\supset \ker \cc^2\supset \ker \cc\supset \image \cc^2\supset 0.
 \end{equation*}
 Set
$\ga^2=\v \ominus \ker \cc^2, \, \ga^1=\ker \cc^2\ominus \ker \cc,\,\gamma^0=\ker \cc, \,  \delta=\ker \cc\ominus \alpha,$
with $\alpha=\image \cc^2$.  As in the proof of Lemma \ref{reducenilorder}, we define
inductively subbundles $\alpha_j$, $\delta_j$ and $\ga^i_{j}$ of $G^{(j)}(\v)$, for $0\leq i\leq 2$ and $j=0,1,\ldots,r$, such that
$\al_0 = \al$, $\delta_0 = \delta$, $\ga_0^i=\ga^i$,
$\ga^0_j=\alpha_j\oplus \delta_j$ and
$G^{(j)}(\v)=\ga^0_{j}\oplus \ga^1_{j}\oplus \ga^2_{j}.$
 By Lemma \ref{ker c^i}(iii), the ranks of $\ga^0,\ga^1,\ga^2$ must be (a) $3,1,1$, giving $\cc$ of rank $2$ or (b) $2,2,1$, giving $\cc$ of rank $3$. From Lemma \ref{ker c^i}(v), some $\alpha_j$ and $\delta_j$ may be zero; however, for $i=1,2$, the $\gamma^i_j$ are all non-zero; in fact the compositions $\ga_0^i \to \cdots \to \ga^i_k$, with $0<k\leq r$ are isomorphisms (away from isolated points).

 We have the following
 refinement of diagram \eqref{diag:second return}:

\vspace{0ex}
\begin{equation}
\begin{gathered}\label{diag:second return3}
\xymatrixrowsep{0.6pc}
\xymatrixcolsep{1.7pc}
\xymatrix{\v & G^{(1)}(\v) &\ldots & G^{(r)}(\v) & R
	 &\v \\
\ga^2  \ar[r] & \ar[r] & \ar[r] & \ga_r^2  \ar[r]\ar[llld]\ar[llldd]\ar[lllddd] & \ar[r]
	 &\ga^2 &
\\
\ga^1  \ar[r]\ar@{-}[u] &  \ar[r]\ar@{-}[u] &  \ar[r]\ar@{-}[u] & \ga_r^1  \ar[r]\ar@{-}[u]\ar[llldd]\ar@{-->}[llld] & \ar[r]\ar@{-}[u]\ar[ru]
	 &\ga^1\ar@{-}[u] &
\\
\delta  \ar[r]\ar@{-}[u] &  \ar[r]\ar@{-}[u] & \ar[r]\ar@{-}[u] & \delta_r  \ar[r]\ar@{-}[u] & \ar[r]\ar@{-}[u]\ar[ruu]\ar[ru]
	 &\delta\ar@{-}[u]&
\\
\alpha \ar[r]\ar@{-}[u] &  \alpha_1\ar[r]\ar@{-}[u] & \ar[r] \ar@{-}[u]& \alpha_r  \ar[r]\ar@{-}[u] &\alpha_{r+1}  \ar[r]\ar@{-}[u]\ar[ruu]\ar[ru]
	 &\alpha\ar@{-}[u]&
}
\vspace{2.5ex}
\end{gathered} \end{equation}
where, as in diagram \eqref{diag:second returnp}, we repeat the first column and set $\alpha_{r+1}=\image A'_{G^{(r)}(\v)}|_{\alpha_{r}}$. Again, not all arrows are shown. No vertical downward arrows exist, which reflects the fact that, for each $i$  and $j$, the
subbundles $\alpha_j$ and $\bigoplus_{k\leq i}\ga^k_{j}$ are holomorphic in $G^{(j)}(\v)$. Moreover:
\begin{enumerate}
  \item From Lemma \ref{ker c^i},
$\alpha$ and $\ga^2$ have rank 1 and $\cc^2:\ga^2\to\alpha$ is an isomorphism (away from isolated points). By Proposition \ref{nilorder p}, $\cc^2\circ \ee|_{\alpha}$ is nilpotent, so zero.  Hence, $\al_{r+1} \to \ga^2$ is zero and no further arrows of the form $\nearrow$ exist besides those shown.
  \item The arrow $\delta_r\rightarrow \alpha$ is zero because $\delta \subset \ker \cc$; no further arrows of the form $\swarrow$ exist  besides those shown.
  \item  None of the arrows $\searrow$ are shown, as they do not affect the following arguments.
 \item Since $\cc^2$ is non-zero, the arrows  $\ga^2_r\rightarrow \ga^1$, $\ga^1_r\rightarrow \alpha$  and the composition $\ga^2 \to \cdots \to \ga^2_r \to \ga^1 \to \cdots \to \ga^1_r \to \alpha$ are all non-zero.
\end{enumerate}

 As above, we have two possible cases: 

\medskip

(a) \textbf{$\image \cc$ has rank $2$}. In this case, $\v \ominus \ker \cc =\ga^1\oplus \ga^2$ also has rank $2$. By Lemma  \ref{ker c^i},  $\delta$ has rank $2$ and  $\ga^1$ has rank $1$. Moreover, the arrow $\ga^1_r\rightarrow \delta$ (shown dashed in diagram \eqref{diag:second return3}) is zero;
 otherwise its image would give a component of $\image \cc^2$ in $\delta$.
Forward replace $\alpha$ to obtain the harmonic map $\wt{\v}$ given by $\wt \v=\alpha_1\oplus\delta\oplus\ga^1\oplus\ga^2$.  {Define $\wt \cc_r$ as in Lemma \ref{reducenilorder} (with  $\beta=\delta\oplus \gamma^1\oplus\gamma^2$).  Since $\wt \cc_r(\ga^2)$ has a non-zero component in $\ga^1$, the
 new harmonic map $\wt \v$ has isotropy order exactly $r$ and, by  Lemma \ref{reducenilorder},  $\wt \cc_r$ is the first return map of $\wt \v$. Moreover,
 since, as we observed above, $\alpha_{r+1}\rightarrow \ga^2$ is zero and, in this case, the arrow $\ga^1_r\rightarrow \delta$ is also zero, we have $(\cc\circ \ee)(\alpha)\subset \alpha$.
 This implies, by  Lemma \ref{reducenilorder}, that $\wt \cc_r^2=0$, that is,  $\wt \cc_r$ has nilorder $2$.}

\medskip

(b) \textbf{$\image \cc$ has rank $3$}.   In this case,  $\ga^1$ has rank $2$ and $\delta $ has rank $1$.  By  Lemma \ref{ker c^i}(iii), the arrow $\ga^1_r\rightarrow \delta$ (the dashed arrow) is non-zero and, since $\cc^2(\ga^2) \subseteq \alpha$, we have
\begin{equation}\label{eq4:nil3}
\image \pi_{\ga^1}\circ \cc|_{\ga^2}=\ker \pi_{\delta}\circ \cc|_{\ga^1}.
\end{equation}
On the other hand, by Proposition \ref{c^2ece}, the cycle
$c^2 \circ e \circ c \circ e$ is zero on $\al$.  This implies that
\begin{equation}\label{eq1:nil3}
\image (\cc\circ \ee|_{\alpha})
=\ker (\cc^2\circ  \ee)\cap \image (\cc\circ \ee|_{\alpha})\subseteq \ker (\cc^2\circ \ee)\cap \{\alpha\oplus\delta\}.
\end{equation}

 Forward replace $\ker \cc=  \alpha\oplus \delta$ to obtain the harmonic map $\wt{\v}:M\to G_5(\C^n)$ given by $\wt \v=\alpha_1\oplus\delta_1\oplus\ga^1\oplus\ga^2$. Define $\wt \cc_r$ as in Lemma \ref{reducenilorder} (with $\beta=\ga^1\oplus\ga^2$).
 Since $0\neq \wt \cc_r(\ga^2) \subseteq  \pi_{\be} \circ \cc(\ga^2) \subseteq \ga^1$, $\wt{\v}$ has isotropy order exactly $r$, with $\wt \cc_r$ as first return map. Observe that
\begin{equation}\label{eq3:nil3}
\image \wt \cc_r\subseteq \image \pi_{\ga^1}\circ \cc|_{\ga^2}+\image \pi_{\ga^1+\ga^2}\circ \ee|_{\alpha+\delta}.
\end{equation}
We claim that $\wt \v$ satisfies $\rk \wt \cc_r < \rk \cc$.

First, suppose that $\cc\circ \ee(\alpha)\subseteq \alpha$. Together with \eqref{eq4:nil3}, this implies that
\begin{equation}\label{eq2:nil3}
\ee(\alpha)\subseteq \alpha\oplus\delta\oplus\ker \pi_{\delta}\circ \cc|_{\ga^1}=\alpha\oplus\delta\oplus \image \pi_{\ga^1}\circ \cc|_{\ga^2}.
\end{equation}
 Then, in view of \eqref{eq3:nil3} and \eqref{eq2:nil3},
 \begin{align*}
 \image \wt \cc_r\subseteq \image \pi_{\ga^1}\circ \cc|_{\ga^2}+\image \pi_{\ga^1+\ga^2}\circ \ee|_{\delta}.
 \end{align*}
 Since $\delta$ and $\ga^2$ have rank $1$, the subbundle on the right has rank at most 2, so that $\image \wt \cc_r$ also has rank at most $2$.

Second, suppose that $\cc\circ \ee(\alpha)\nsubseteq \alpha$.
 By \eqref{eq1:nil3}, this implies that $\cc^2\circ \ee(\delta)= 0$,
that is, $\ee(\delta) \subseteq \v \ominus \ga^2$.
Then, in view of \eqref{eq3:nil3},
 $\image \wt \cc_r\subset
 \ga^1$. Since the subbundle  $\ga^1$ has rank 2,
  $\image \wt \cc_r$ again has rank at most $2$.

Thus, in both cases, forward replacing $\ker \cc$ decreases the rank of the first return map, and  thus establishes our claim.
\end{proof}

Putting all the above work together we obtain the following generalization of \cite[Theorem 4.2]{burstall-wood} to arbitrary Riemann surfaces; we note that, for $k=2$, Theorem \ref{BuWo 3.3} gives a more  precise result.

\begin{theorem} \label{BuWo 4.2}
Let $\v:M \to G_k(\C^n)$ be a harmonic map of finite
uniton number from a Riemann surface, $k=2,3,4,5$.
Then there is a sequence of harmonic maps $\v_0,\ldots,\v_N:M \to G_k(\C^n)$ such that
\begin{enumerate}
\item[{\rm (i)}] $\v_0$ is either holomorphic or there is a harmonic map $\psi:M \to G_t(\C^n)$, $1 \leq t < k$ and a holomorphic subbundle $\alpha$ of $\bigl(\psi \oplus G''(\psi)\bigr)^{\perp}$ such that $\v_0 = \psi \oplus \alpha$$;$
\item[{\rm (ii)}] $\v_N = \v;$
\item[{\rm (iii)}] For each $i$, $0 \leq i < N$, there is a holomorphic subbundle $\ga_i$ of $\v_i$ such that $\v_{i+1}$ is obtained from $\v_i$ by forward replacement.
\end{enumerate}
\end{theorem}

\begin{proof}
Using one of the propositions above, if $\v$ is of finite uniton number and 
$\pa'$-irreducible, we can forward replace a suitable subbundle of it to give a map which reduces the rank of $\cc$, decreases the nilorder of $\cc$, or increases the isotropy order of $\v$.
We continue this procedure until we reach a $\pa'$-reducible harmonic map.  

Reversing the orientation on the domain, this becomes $\pa''$-reducible and the replacements become backward replacements.  Inverting these gives forward replacements leading to the above result.   The $\pa''$-reducible map is described by \cite[Theorem 4.1]{burstall-wood}.

More specifically, for $\v$ of finite uniton number,
at each stage:

(i) If $\cc$ has nilorder $2$, for $k=2$ we apply Proposition \ref{G2}.
For arbitrary $k$, (a) if $\rk c = 1$, we apply Proposition \ref{Gk nil 2 rk 1}; (b)  if $\rk c=2$, then, by Lemma \ref{ker c^i}(iv), $k \geq 4$, and we apply Proposition \ref{G4} for $k=4$ or Proposition \ref{G5} for $k=5$.

(ii) If $\cc$ has nilorder $p > 2$, then Proposition \ref{Gk nil k,k-1} covers $p=k$ or $k-1$. The only remaining case is $k=5$ and $\cc$ of nilorder $3$, which is covered by Proposition \ref{G5 nil 3}.

Lastly, note that, if $\v$ is strongly isotropic, i.e., of infinite isotropy order,
by dimension considerations, some $\pa'$-Gauss bundle must be $\pa'$-reducible and the same result holds with each $\ga_i = \v_i$.
\end{proof}

\section{Applications}
\subsection{Constant curvature minimal surfaces} \label{subsec:const-curv}

We give an application to the important question of finding harmonic maps, or minimal immersions, of constant curvature in various symmetric spaces, here the complex quadric $Q_{n-2}$ or the Grassmannian, $G_2(\R^n)$, of $2$-dimensional subspaces.   Recall that
$Q_{n-2} = \bigl\{[Z] = [Z_1,\ldots,Z_n] \in \CP^{n-1}: Z_1^{\,2} + \cdots + Z_n^{\,2} = 0 \bigr\}$; it can be identified with the real Grassmannian of {\emph{oriented}} $2$-dimensional subspaces of $\R^n$, and so it double covers
$G_2(\R^n)$.  The double covering $Q_{n-2} \to G_2(\R^n)$ is given by $[Z]=[X+\ii Y] \mapsto \spa\{X,Y\}$; when $G_2(\R^n)$ is considered as a (totally geodesic) subspace of $G_2(\C^n)$, this span equals $\spa\{Z, \ov{Z}\}$.

Harmonic maps from $S^2$ to $G_2(\R^n)$ were studied in \cite{bahy-wood-G2}.  Sections 1--4 of that paper depend on nilpotency of the first return map so immediately generalize to harmonic maps of finite uniton number from any Riemann surface $M$.
In particular,
recall from \S \ref{subsec:nilorder 2} that a `real mixed pair' is a map $\v:M \to G_2(\R^n)$ of the form $h \oplus \ov{h}$ where $h:M \to Q_{n-2}$ is holomorphic; thus $\v$ is the projection of that holomorphic map under the above double covering, which confirms that any real mixed pair is harmonic.
For brevity, write $h_i = G^{(i)}(h)$ for $i=0,1,\ldots$.
{}From \cite[Lemma 2.14]{bahy-wood-G2}, $\v$ has finite isotropy order $r$ (necessarily odd \cite[Proposition 2.8]{bahy-wood-G2}) if and only if $h$ satisfies the condition
\begin{equation} \label{real-isotropy-order}
h_i \perp \ov{h} \ (0 \leq i \leq r)
\text{ and } h_{r+1} \not\perp \ov{h};
\end{equation}
the integer $r$ is sometimes called the \emph{(real) isotropy order of $h$} \cite[\S 4.2]{FSW-On}; if there is no such $r$, $\v$ has infinite isotropy order;  this can only happen if $h$ is \emph{not (linearly) full}, i.e.\  its image lies in a proper projective
subspace of $\CP^{n-1}$, cf.\ Proposition \ref{reducible-G2n}(b) below.

For $M = S^2$, an algorithm involving integration for determining such holomorphic maps $h$, and so \emph{all real mixed pairs of finite isotropy order}, is given in \cite[\S 5]{bahy-wood-G2}.

For a \emph{full} holomorphic map $h:M \to \CP^{2p}$ the maximum value of $r$ in \eqref{real-isotropy-order} is equal to $2p-1$ and then $h$ is
called \emph{totally isotropic} \cite[Definition 3.13]{eells-wood}.  In this case $h_{2p-i} = \ov{h_i}$ \ $(i=0, \ldots, 2p)$; $h_p = \ov{h_p}$ has values in $\RP^n$ and, by being careful with orientations (cf. \cite[\S 3C]{eells-wood}), lifts to a full (strongly) isotropic harmonic map $\psi$ into $S^n$.  The assignment $h \mapsto \psi$ is the $2:1$-correspondence of E.~Calabi \cite{calabi-spheres, calabi-quelques};
see \cite[\S 4.2]{FSW-On} for connections with harmonic maps with $S^1$-invariant (see \S \ref{sec:finite type}) extended solutions into the orthogonal group, leading to explicit constructions of totally isotropic holomorphic maps from Riemann surfaces which do not use integration.

For a harmonic map $\v:M \to G_k(\R^n)$, we have $G''(\v) = \ov{G'(\v)}$; hence $\v$ is $\pa'$-reducible if and only if it is $\pa''$-reducible, we shall simply say that $\v$ is \emph{reducible}.
In \cite[Proposition 2.12]{bahy-wood-G2}, it is shown that {\emph{a harmonic map of finite isotropy order from $S^2 \to G_2(\R^n)$ is reducible if and only if it is a real mixed pair}}; as above, this generalizes to harmonic maps of finite uniton number from any Riemann surface.

In \cite{JiaoLi Qn}, X.~Jiao and M.~Li study harmonic maps from the $2$-sphere to $G_2(\R^n)$ of constant curvature; in their  Proposition 3.1, they find all such strongly isotropic reducible maps. We generalize this to harmonic maps of finite uniton number from a Riemann surface; {for our first result, Proposition \ref{reducible-G2n}} below, we do not require constant curvature.  For a complex vector space $W$, let $P(W)$ denote the corresponding projective space; if $W$ has dimension $m$ then $P(W)$ may be identified with $\CP^{m-1}$.
Also, for $v=(v_1,\ldots,v_n),
w = (w_1, \ldots, w_n) \in \C^n$ we let $\langle v, w \rangle$ denote the standard Hermitian inner product
$v_1 \ov{w_1} + \cdots + v_n \ov{w_n}$.

\begin{proposition} \label{reducible-G2n}
Let $\v: M \to G_2(\R^n)$ be a non-constant
reducible harmonic map which is strongly isotropic, i.e., of infinite isotropy order.
Then, one of the following holds:

{\rm (a)} $\v = f_p \oplus \underline{\kappa}$ for some
$p \in \{1,2,\ldots\}$ where
$f:M \to P(W) \subset \CP^n$ is a full totally isotropic holomorphic map with $W$ a constant real ($W=\overline{W}$)  subspace of\/ $\C^n$ of  dimension $2p+1$, and $\underline{\kappa}$ is a constant real $1$-dimensional subspace of\/ $\C^n$ orthogonal to $W$.

{\rm (b)} $\v$ is a real mixed pair:  $\v= h \oplus \ov{h}$ where $h: M \to P(W) \subset \CP^{n-1}$ is a full holomorphic map with $W$ an isotropic subspace of\/ $\C^n$ (here `isotropic' means that $W \perp \ov{W}$).

Conversely, any map  $\v:M \to G_2(\R^n)$ of type (a) or (b) is a strongly isotropic reducible harmonic map.
\end{proposition}

\begin{proof}
As in the proof of \cite[Proposition 3.1]{JiaoLi Qn}, for some
$p \in \{1,2,\ldots\}$,

\begin{equation} \label{fp}
\text{\emph{there is a holomorphic map }} f:M \to \CP^{n-1}
\text{\emph{ such that $f_p$ lies in $\v$.}}
\end{equation}

 Indeed, by reducibility of $\v$, $G''(\v)$ has rank 1, so is a harmonic map $M \to \CP^{n-1}$ of finite uniton number.  By \cite[Theorem 4.12]{APW1} or Theorem \ref{th:wolfson} below, this is $f_{p-1}$ for some holomorphic map $f:M \to \CP^{n-1}$ and integer $p \geq 1$.  Since  $A''_{\v}:\v \to f_{p-1}$ is surjective, its adjoint (up to sign) $A'_{f_{p-1}}$ is injective with non-zero image $f_p \subset \v$. By reality of $\v$, $\ov{f_p}$ is also non-zero and lies in $\v$.  It follows that $G^{(-i)}(\v) = f_{p-i}$ and $G^{(i)}(\v) = \ov{f_{p-i}}$ \ $(i=1,\ldots, p)$   and these are all of rank $1$.

Now $f_{p+1} \subset G'(\v)  =  \ov{f_{p-1}}$, so either
(a) $f_{p+1} = \ov{f_{p-1}}$ or (b) $f_{p+1} = 0$.

In case (a), $f_p = G'(f_{p-1}) = \ov{G''(f_{p+1})} = \ov{f_p}$ which implies that
$f_i = \ov{f_{2p-i}}$ \ $(i=0,1,\ldots, 2p)$. {It follows that} the $f_i$ \ $(i=0,1,\ldots, 2p)$ span a constant real
 subspace $W$ of  dimension $2p+1$ and $f:M \to P(W)$ is a full totally isotropic map.  Write $\v = f_p \oplus \underline{\kappa}$ (orthogonal direct sum)
  where $\underline{\kappa}$ is a real {rank $1$} subbundle of $\CC^n$.  Then by strong isotropy of $\v$, $\underline{\kappa}$ is orthogonal to $f_i$ for all $i$ and so is in the orthogonal complement of $W$.  It must be constant, otherwise it would contribute an additional subspace to $G'(\v)$.

In case (b), set $h = \ov{f_p}$, then $h$ is holomorphic.
Further, choose a holomorphic section  $F_0$ of $f$ and set
$F_i = A'_{f_{i-1}}(F_{i-1})$ for $i=1,2,\ldots$; then,  {since $F_{p-1}$ lies in $G''(\v)$ and $\ov{F_p}$ in $\v$,} their Hermitian inner product $\langle F_{p-1},\ov{F_p} \rangle$ is zero.  Differentiating this with respect to $z$ gives
$\langle F_p,\ov{F_p} \rangle + \langle F_{p-1},\ov{F_{p+1}} \rangle = 0$. Since $f_{p+1} = 0$,  $A'_{f_p}$ and so $F_{p+1}$ are zero, so that the second inner product is zero;
hence the first inner product is also zero,
showing $h$ is orthogonal to $\ov{h}$
so that $\v = h \oplus \ov{h}$ is a real mixed pair.
Set $W = \spa\{h_i:i=0,1,\ldots,p\}
= \ov{\spa\{f_j:j=0,1,\ldots,p\}}$.  Since $W$ is closed under differentiation with respect to $z$ and $\zbar$, it is constant; by strongly isotropy of $\v$ and orthogonality of $h$ and $\ov{h}$, $W$ is isotropic.

The converse statement is clear.
\end{proof}

\begin{remark} \label{rem:reducible-G2n}
{\rm (i)} A real mixed pair satisfies \eqref{fp} for $p=0$, as well as for some $p \geq 1$, but we need $p \geq 1$ in the proof.

\smallskip

{\rm (ii)} The proof shows that, if $\v: M \to G_2(\R^n)$ is a non-constant \emph{strongly isotropic} reducible map, then $n \geq 2p+2$ with equality if $\v$ is full. In particular, $n \geq 4$; moreover \cite{erdem-wood}, \emph{for a strongly isotropic full reducible harmonic map $M \to G_2(\R^n)$, $n$ is even.} 

When $\v$ is full, we have the following$:$ 

\indent\indent  \emph{In case (a),} $\underline{\kappa}$ is the constant $1$-dimensional subspace $W^{\perp}$.
Letting $e_1,\ldots, e_n$ denote the standard orthonormal basis $(1,0, \ldots, 0),$ $\ldots,$ $(0,0,\ldots, 1)$ of $\C^n$; up to an isometry in $\U(n)$ we can take $W = \spa\{e_1,\ldots, e_{n-1}\}$ so that $W^{\perp} = \spa\{e_n\}$.

\indent\indent \emph{In case (b),} $W$ is a \emph{maximal} (so $(p+1)$-dimensional) isotropic subspace of\/ $\C^{2p+2}$. 
 Up to an isometry in $\U(n)$, we can take $W$ to be
the maximally isotropic subspace of\/ $\C^n$ given by the image of the isometric linear {injection $I:\C^{p+1} \to \C^n$},
$e_i \mapsto (1/\sqrt{2})(e_{2i-1}+\ii e_{2i})$,
from $\C^{p+1} = \spa\{e_1,\ldots,e_{p+1}\}$.

\smallskip

 {\rm (iii)} There are, however, non-constant reducible harmonic maps $M \to G_2(\R^n)$ for $n=3$; in fact, 
 \emph{all non-constant harmonic maps of finite uniton number are real mixed pairs $f \oplus \ov{f}$ where $f:M \to Q_1 \cong \CP^1$ is holomorphic;} these have isotropy order $1$,
cf. \cite[Proposition 6.4]{bahy-wood-G2}.
 \end{remark}

Now we see which of the above maps are of constant curvature.
Recall that the \emph{(round) Veronese maps} are the holomorphic maps $V_0^{(m)}:  S^2 \to \CP^m$ \ $(m=1,2,\ldots)$ given by the formula
$$
S^2 = \C\cup\infty \ni z \mapsto \Big[1, \sqrt{\binom{m}{1}} z, \ldots, \sqrt{\binom{m}{r}} z^r, \ldots, z^m \Big].
$$
For each $m$, this is an {embedding} which induces a metric of constant curvature on $S^2$, as do its Gauss bundles $V_p^{(m)} := G^{(p)}(V_0^{(m)}):S^2 \to \CP^m$
 \ $(p=0,1,\ldots, m)$ {(see \S \ref{subsec:more} for `Gauss
 bundles' and \cite{BJRW} for explicit formulae)}; these are also embeddings except when $m=2p$.

When $m=2p$, $V_p^{(m)}:S^2 \to \CP^m$ is {isometric, via an isometry of $\CP^m$, to an immersion $S^2 \to \RP^m \subset \CP^m$ called a \emph{Veronese--Bor\r{u}vka map}, see  \cite{boruvka,BJRW}. This lifts to an immersion $S^2 \to S^m$ of constant curvature given by an orthonormal basis of spherical harmonics, see \cite{wallach}}.
 As in \cite[\S 3]{JiaoLi Qn}, the isometry {on $\CP^m$} is given by any $U \in \U(2p+1)$ satisfying the condition
\begin{equation} \label{UTU}
U^T U = W_0\,, \hspace {2ex} \text{ equivalently, } \hspace {2ex} \ov{U} = U W_0	
\end{equation}
where $W_0 = \antidiag(1,-1,1,-1, \ldots,1)$;
this ensures that $U V_0^{(m)}$ is totally isotropic.
Up to left-multiplication by an orthogonal matrix, $U$ is given by a matrix $U_0$ defined in \cite[Remark 3.3]{JiaoLi Qn}; for $p=2$ this matrix is
$$
U_0 = \begin{pmatrix}  1/\sqrt{2} & 0 & 0 & 0 & 1/\sqrt{2} \\
				\ii/\sqrt{2} & 0 & 0 & 0 & -\ii\sqrt{2} \\
				0 & 1/\sqrt{2} & 0 & -1/\sqrt{2} & 0 \\
			    0 & \ii/\sqrt{2} & 0 & \ii/\sqrt{2} & 0 \\
					            0 & 0 & 1 & 0 & 0  \end{pmatrix}\,,
$$ and the totally isotropic holomorphic map $f= U_0 V_0^{(4)}:S^2 \to \CP^4$ is given by
$$
f(z) = [1+z^4, \ii(1-z^4), 2(z-z^3), 2\ii(z+z^3),2\sqrt{3}z^2].
$$
The map {$\v:= G^{(2)}(f):S^2 \to \RP^4$ lifts to an immersion $S^2 \to S^4$ which} defines a minimal surface of constant curvature, cf.\ \cite[p.~149]{chern}  where a formula in real Cartesian coordinates is given.

We now generalize \cite[Propositions 3.2 and 3.5]{JiaoLi Qn} to an arbitrary Riemann surface $M$; as above, $e_1,\ldots, e_n$ denotes the standard orthonormal basis $(1,0, \ldots, 0),$ $\ldots,$ $(0,0,\ldots, 1)$ of $\C^n$ and $V_p^{(m)}:S^2 \to \CP^m$ denotes the $p$th Gauss bundle of the Veronese map.

\begin{theorem} \label{jiao-li}
Let $\v: M \to G_2(\R^n)$ be a full reducible immersive harmonic map of finite uniton number such that the induced metric on $M$ has constant
Gauss curvature. Then, for any $q \in M$ there is {an open neighbourhood $A$ of $q$ and an isometry $\iota$ from $A$ to an open} subset of $S^2$ such that $\v|_A = \Phi \circ \iota$ where $\Phi: S^2 \to G_2(\R^n)$ is one of the following maps$:$

{\rm (i)} If\/ $\v$ has finite isotropy order $r$, then
 $\Phi = F \oplus \ov{F}$ where $F:S^2 \to \CP^{n-1}$ is a holomorphic map given by
$F= U V_0^{(n-1)}$ for some $U \in \U(n)$.  Further, if $r \geq n-4$, then $n$ is odd, $r=n-2$, and $U$ satisfies \eqref{UTU}
so that $F$ is totally isotropic.

{\rm (ii)} If\/ $\v$ has infinite isotropy order,
then $n$ is even and writing $n=2p+2$, up to an isometry of $G_2(\C^n)$ given by an element of $\U(n)$,
\emph{either}

\indent\indent {\rm (a)} $\Phi = F_p \oplus \spa\{e_n\}$, where
$W = \spa\{e_1,\ldots, e_{n-1}\}$ and
$F_p:S^2 \to \RP^{2p} \subset \CP^{2p} = P(W)$ is given by
$F_p = U V_p^{(2p)}$ for some $U \in  \U(2p+1) $ satisfying \eqref{UTU},
\emph{or}

\indent\indent {\rm (b)} $\Phi = F \oplus \ov{F}$,
where $W$ is
the maximally isotropic subspace of\/ $\C^n$ given by the image of the isometric linear injection $I:\C^{p+1} \to \C^n$,
$e_i \mapsto (1/\sqrt{2})(e_{2i-1}+\ii e_{2i})$,
 and $F: S^2 \to P(W)$ is the Veronese map
$V_0^{(p)}:S^2 \to \CP^p$ followed by the isometry
$\CP^p = P(\C^{p+1}) \to P(W)$ induced by $I$.
\end{theorem}

\begin{proof}
In case (i), as explained at the start of this section, results in Sections 1--4 of \cite{bahy-wood-G2} extend to harmonic maps of finite uniton number.  In particular \cite[Proposition 2.12]{bahy-wood-G2} extends to show that $\v$ is a real mixed pair $f\oplus \ov{f}$.  Clearly $\v$ is full and has constant curvature if and only if $f$ has these properties.  By a theorem of E.~Calabi (see \cite[Theorem 1.1]{lawson}, this happens if and only if, locally, $f$ is isometric to a Veronese map $V_0^{(m)}$; by fullness $m=n-1$. The statement for $r \geq n-4$ is proved in \cite[Proposition 3.2]{JiaoLi Qn}.

\medskip

In case (ii), $\v$ is described by case (a) or (b) of Proposition \ref{reducible-G2n} with choice of $W$ as in Remark \ref{rem:reducible-G2n}(ii).

\indent\indent In case (ii)(a), $f_p$ has values in $\RP^{2p}$; in fact, by taking account of orientations, $f_p$ defines a minimal isometric immersion
 $M \to S^{2p}$ \cite[\S 3C]{eells-wood}; it is of constant curvature since $\v$ is. If that curvature is positive, then, by Minding's theorem, each point of $M$ has a connected open neighbourhood $A$ and an isometry $\iota:A \to \iota(A)$ to a connected open subset of $S^2$. Then, by a result of N.~R.~Wallach \cite[Theorem 1.1]{wallach},
cf.\ R.~L.~Bryant \cite[Theorems 1.5 and 1.6]{bryant},
{$f_p|_A = G \circ \iota$ where $G$ is is isometric to the Veronese--Bor\r{u}vka map given above.
(See \cite[Theorem 5.2]{calabi-spheres} and
\cite[p.~103]{doCarmo-W} for related results.)}

If the constant curvature is non-positive then, by
Bryant \cite[Theorems 2.3 and 3.1]{bryant},
 it is zero and for a small enough open set $A$, there is an isometry $\iota$ from $A$ to an open subset of $\C$ such that $f_p = H\circ \iota$ where $H$ is a harmonic map on $\C$ given by exponential formulae.
From these formulae, it is clear that all 
$G^{(m)}(H)$ are non-zero.  Then, for dimension reasons, there must be some $m$ such that $G^{(m)}(H)$ is not perpendicular to $H$, so $H$ is of finite isotropy order, contradicting the hypothesis for part (ii) that  $\v$ and so $H$ is of infinite isotropy order. 

\indent\indent In case (ii)(b), $\v$ has constant curvature if and only if the holomorphic map $f:M \to P(W) \cong \CP^p$ does;
again by
\cite[Theorem 1.1]{lawson}, this holds if and only if  $F$ is  locally isometric to a Veronese map.
\end{proof}

\begin{remark} 
\begin{enumerate}
\item[(i)] In case (i) with $r \geq n-4$ and case (ii)(a), we can take $U$ to be the matrix $U_0$ described above.

\item[(ii)] The theorem holds \emph{globally}, i.e., we can take $A = M$, if $M$ together with its induced metric, is isometric to an open connected subset of $S^2$. In particular, the Riemann surface $M$ is conformally equivalent to an open connected subset of the Riemann sphere; according to  \cite{wiki}, this holds if and only if $M$ is conformally equivalent to the Riemann sphere or to the complex plane with  finitely many closed intervals parallel to the real axis removed, such an $M$ is called \emph{planar} or \emph{schlichtartig}.
\end{enumerate}
\end{remark}

\subsection{Finite uniton vs finite type} \label{sec:finite type}
As in \cite{APW1}, we use Segal's \emph{Grassmannian model} \cite{segal}; this associates to an extended solution $\Phi$  of a harmonic map $\v:M \to \U(n)$  the family of closed subspaces $W(z),~z\in M,$  of the Hilbert space  $L^2(S^1,\C^n)$, defined by
\begin{equation}\label{W-def}
W(z)=\Phi(\cdot,z)\H_+,
\end{equation}
where  $\H_+$ is the \emph{Hardy space} of $\C^n$-valued functions, i.e., the closed subspace of  $L^2(S^1,\C^n)$ consisting of Fourier series whose negative coefficients vanish.
The subspaces $W(z)$ form the fibres of a smooth subbundle $W$ of the \emph{trivial bundle} $\HH := M \times L^2(S^1,\C^n)$ over the Riemann surface, see, for example, \cite[\S 3.1]{APW1});
we abbreviate \eqref{W-def} to $W= \Phi\H_+$.
For a characterization of the bundles $W$ which occur as $\Phi\H_+$, see \cite[\S 7.1]{pressley-segal}, or \cite[\S 3.1]{APW1} for an alternative approach; for such a $W$ there is a \emph{unique} $\Phi$ such that $W = \Phi\H_+$\,.

We denote by $S$ the \emph{forward shift} on $L^2(S^1,\C^n)$ defined by
$$
(Sf)(\la)=\la f(\la), \qquad (f \in L^2(S^1,\C^n), \ \la\in S^1),
$$
and by $\pa_z$ and $\pa_{\bar{z}}$ differentiation with respect to  $z$ and $\bar{z}$ respectively, where $z$ is a local  (complex) coordinate on $M$; note that all equations below are independent of the choice of local coordinate.
The extended solution equation \eqref{extsol} for $\Phi$ is equivalent to the pair of equations  for $W = \Phi\H_+$:
\begin{equation} \label{W-eq}
S\, \pa_z \Gamma(W)\subseteq \Gamma(W) \quad \text{ and } \quad
\pa_{\bar{z}}\Gamma(W)\subseteq \Gamma(W)\,,
\end{equation}
see \cite{segal, guest-book};  we say that \emph{$W$ is an extended solution} if it satisfies these equations. Note that the second equation says that $W$ is a holomorphic subbundle of $\HH$.

  Let $W(z)=\Phi(\cdot,z)\H_+$ be an extended solution associated to a harmonic map $\v:M\to \U(n)$.  By filling out zeros, $$W_{(i)}:=W+\pa_z W+\cdots+\pa^i_zW,\quad i=0,1,\ldots$$
 gives a nested sequence of  holomorphic subbundles of $\HH_+$, which is called the \emph{Gauss sequence of\/ $W$} \cite{APW1,svensson-wood-unitons}.   The subbundles $W_{(i)}$ clearly satisfy \eqref{W-eq}, and so are extended solutions; these give a sequence of harmonic maps $\v_i:M\to \U(n)$ called the \emph{Gauss sequence of $\v = \v_0$}.
 In fact, choosing the minus sign in \eqref{uniton} as in \S \ref{subsec:Grass}, the harmonic map
$\v_{i}$ is obtained from $\v_{i-1}$ by adding the uniton $\image {A^{\v_{i-1}}_z}$, see \cite{APW1,svensson-wood-unitons}.
Then the finiteness criterion, Proposition \ref{S-stabilizes}, reads as follows: \emph{a harmonic map $\v$ is of finite uniton number if and only if its Gauss sequence  terminates, i.e., some $\v_i$ is constant \cite{APW1}}.

If, now, $\v$ is a harmonic map from $M$ to $G_k(\C^n)$, from {\cite[Lemma (2.2) and Remark (2.6)(iv)]{wood-Un}}, we see that \emph{the harmonic sequence  $\bigl(G^{(i)}(\v)\bigr)$ of $\v$ coincides with its Gauss sequence if $\v:M\to  G_k(\mathbb{C}^n)$ is $\pa''$-irreducible}. This leads to the following generalization  of a well-known result by J. G. Wolfson \cite[Theorem 3.6]{wolfson} that states that \emph{the harmonic sequence of a harmonic map from $S^2$ to $G_k(\C^n)$ terminates} (i.e., $G^{(i)}(\v)$ is zero for some $i$).

\begin{theorem} \label{th:wolfson}
A harmonic map $\v:M\to G_k(\C^n)$  is of finite uniton number if and only if its harmonic sequence terminates.
\end{theorem}
\begin{proof}
Again, choosing the minus sign in \eqref{uniton},  the  Gauss bundle $G^{(i)}(\v)$ is obtained from $G^{(i-1)}(\v)$ by adding the uniton $\alpha$, with $\alpha =G^{(i-1)}(\v)\oplus G^{(i)}(\v)$, as seen in {\S\ref{subsec:Grass}}.  Hence, if the harmonic sequence terminates, then we can obtain the constant map by adding a finite number of unitons to $\v$. This means that $\v$ is of finite uniton number.

The converse result follows directly from our previous discussion when $\v$ is $\pa''$-irreducible. Suppose that $\v$ is  of finite uniton number but not  $\pa''$-irreducible. Clearly,
the first Gauss bundle $G^{(1)}(\v) = G'(\v)$ is also of finite uniton
 number; moreover,  $G^{(1)}(\v)$ is $\pa''$-irreducible since
  $A''_{G^{(1)}(\v), \v}$ is minus the adjoint of the surjective map
   $A'_{\v,G^{(1)}(\v)}$. So the Gauss sequence of $G^{(1)}(\v)$
    terminates and equals its harmonic sequence,  hence the
     harmonic sequence of $\v$ also terminates.
\end{proof}

A harmonic map is said to be of \emph{finite type} if it admits an  associated extended solution of the form $W=\exp(\xi z)\H_+$, with $\xi=\sum_{i=-1}^d\xi_k\lambda^k$ constant with respect to $z$. The harmonic maps of finite type can also be obtained by using integrable systems methods from a certain Lax-type equation \cite{BurstallPedit,guest-book} and they play an important role in the theory of harmonic maps from tori into symmetric spaces. For example, it is known (see \cite{pacheco-tori} and references therein) that \emph{all non-constant harmonic tori in the $n$-dimensional Euclidean sphere $S^n$ or the complex projective space $\CP^n$ are either of finite type or of finite uniton number}.

An extended solution $\Phi$ and the corresponding $W = \Phi\H_+$ is called \emph{$S^1$-invariant} if $f \in W$ implies that $f_\mu \in W$ for all $\mu \in S^1$, where $f_\mu(\lambda) :=f(\mu\lambda)$, \ $\la \in S^1$.  For example, holomorphic and anti-holomorphic maps into a (complex) Grassmmanian admit $S^1$-invariant extended solutions; conversely, if $\Phi$ is $S^1$-invariant, $\v = \Phi_{-1}$ maps into a Grassmannian.
In \cite[Corollary 1]{pacheco-tori} it was proved that
\emph{no non-constant harmonic map $\v:T^2\to G_k(\mathbb{C}^n)$ associated to an $S^1$-invariant extended solution is of finite type}.   We now  generalise this result.

\begin{theorem} \label{finite type constant}
  Let $\v:T^2\to G_k(\mathbb{C}^n)$ be a harmonic map simultaneously of finite uniton number and finite type. Then $\v$ is constant.
\end{theorem}
\begin{proof}
  Let $\v:T^2\to G_k(\mathbb{C}^n)$ be a harmonic map of finite uniton number. Then the sequence
  $\v, G^{(1)}(\v), G^{(2)}(\v),\ldots$ of Gauss bundles terminates. Let $\psi:T^2\to G_r(\mathbb{C}^n)$ be the last non-zero harmonic map in this sequence, with $r\leq k$. This is an anti-holomorphic map, hence $\psi$ admits an $S^1$-invariant extended solution.

 If $\v$  is also of finite type, then any Gauss bundle $G^{(i)}(\v)$ gives a harmonic map of finite type (see \cite[Theorem 9.2]{correia-pacheco}). In particular, besides admitting an $S^1$-invariant extended solution, $\psi$ is of finite-type. By \cite[Corollary 1]{pacheco-tori} as stated above,  $\psi$ must be constant.
But any  non-zero Gauss bundle $G^{(i)}(\v)$, $i \geq 1$, is $\pa''$-irreducible as above, so
 % By the definition of Gauss bundles, this implies that
$\v=\psi$. Hence $\v$ is constant.
\end{proof}

\end{document}